\pdfoutput=1
\newif\ifconf
\conftrue

\newif\ifsub
\subtrue

\ifconf
\documentclass[11pt,hidelinks]{article}
\usepackage{fullpage}
\else
\documentclass[12pt]{amsart}
\fi 

\input{macros.tex}

\title{\textbf{On the Pseudo-Mixing of Kac's Walk}} 

\author{Natesh S. Pillai\\Harvard
\and Aaron Smith\\uOttawa and TIMC
\and Vinod Vaikuntanathan\\MIT}

\begin{document}
\maketitle

\begin{abstract}

We study the pseudo-mixing properties of Kac's walk, a popular Markov chain Monte Carlo algorithm introduced by Hastings in 1970 to sample from the Haar measure on $\So{n}$. The mixing time of the full walk is known to be $n^{2} \, \mathrm{polylog}(n)$ in Wasserstein distance and very recently, in total variation distance (Oliveira, Annals of Applied Probability, 2009; and Pillai and Smith, Annals of Probability, 2018, and Arxiv, 2026).  

However, algorithms based on replacing Haar-distributed matrices with ones generated by Kac's walk for a mere $n \, \mathrm{polylog}(n) \ll n^{2}$ steps (which we will refer to as Kac matrices) have enjoyed strong empirical success achieving near-quadratic runtime speedup (e.g. Ailon and Chazelle, STOC 2006; and Jain et al., Annals of Applied Probability, 2022). Indeed, a recent work (Vaikuntanathan and Zamir, SODA 2025) conjectured that such matrices are {\em computationally indistinguishable} from Haar-random matrices, meaning that any polynomial-time algorithm using Haar matrices would work just as well using Kac matrices.

We show two results supporting this conjecture:

\begin{itemize}
    \item First, we prove a conjecture of Oliveira that the first $k$ columns of Kac's walk mix in time $nk \, \mathrm{polylog}(n)$ uniformly in $n$ and $k$;
    \item Secondly, we use this and a representation-theoretic argument to show that low-degree polynomials cannot distinguish Kac matrices from Haar matrices. In particular, degree-$k$ polynomials fail to distinguish $nk^2\, \mathrm{polylog}(n)$-step Kac matrices from Haar matrices.
\end{itemize}

As an illustrative application, we also show that our second result immediately implies that Kac's walk can be used as a plug-in replacement for Haar matrices in the fast Johnson-Lindenstrauss algorithm, providing a quick proof of a conjecture of Ailon and Chazelle (STOC 2006).  

We believe our results may be of independent interest to MCMC practitioners: they suggest that the ``computational" mixing time is drastically smaller than the mixing time in ``traditional" norms. While it is well-known that this phenomenon is \textit{possible} (Rabinovich, Ramdas, Jordan and Wainwright, Bayesian Analysis, 2020), this is the first illustration of a drastic difference for useful functions of a non-toy chain that we are aware of. 

\end{abstract}

\pagenumbering{arabic}

\section{Introduction}

In a landmark paper, Hastings~\cite[Section~3]{hastings70} introduced a discrete-time random walk $\{X_t\}_{t\geq 0}$ on $\So{n}$,\footnote{$\So{n}$ is the special orthogonal group, the set of all matrices in $\mathbb{R}^{n\times n}$ with determinant $1$.} now called \textit{Kac's walk}, that evolves by the rule
\begin{equation} \label{EqKacWalkDef}
X_{t+1} = R_{i_t j_t}(\theta_t)\, X_t,
\end{equation}
where
\[
(i_t,j_t) \in \{(i,j):1\le i<j\le n\}
\]
are chosen uniformly and independently, \(\theta_t\) are sampled from the uniform distribution on $[0,2\pi)$, and 
\[
R_{ij}(\theta)
=
I_n
+
(\cos\theta-1)(E_{ii}+E_{jj})
+
\sin\theta\, E_{ij}
-
\sin\theta\, E_{ji}
\]
is simply rotation by angle $\theta$ in the $(i,j)$-coordinate hyperplane. That is, starting from a fixed matrix $X_0$, the random process repeatedly picks uniformly random rows $i_t$ and $j_t$ at time $t$, performs a random rotation in the two-dimensional space they span, and replaces rows $i_t$ and $j_t$ with their rotated versions.

It is straightforward to check that, for any $1 \leq k \leq n$, the projection of Kac's walk to the first $k$ columns is again a Markov chain. This is called the $k$-column walk. The special case $k=1$, which gave the walk on $\So{n}$ its name, was introduced in the statistical physics literature in Kac's insightful 1959 paper~\cite{Kac1959ProbabilityRelatedTopics} in an attempt to simplify Boltzmann’s equation.

\paragraph{Previous Work on Mixing Times.}
The stationary distribution of the Kac walk is the Haar measure on $\So{n}$. Many notions of ``mixing times" have been studied extensively for Kac's walk. Early work focused on the relaxation time (\textit{i.e.} inverse of the spectral gap of the transition operator), a rather weak notion of convergence. The first polynomial-time estimate appeared in Diaconis and Saloff-Coste~\cite{DiaconisSaloffCoste2000Kac}, who showed that the relaxation time was $O(n^{3})$. The relaxation time was shown to be $\Theta(n)$ several years later in \cite{CCL03}, and \cite{Mas03} computed the full spectrum of the transition operator. Although Kac's walk is a random walk on a group, it is not conjugation-invariant, and so even complete spectral information does not automatically lead to good estimates of the mixing time in stronger norms. The Wasserstein mixing time was shown to be $O(n^{2} \log(n))$ by Oliveira~\cite{Oliveira_2009}, who also showed that this bound is tight modulo logarithmic factors. Finally, there has been substantial effort to estimate the mixing time in total variation. Diaconis and Saloff-Coste~\cite{DiaconisSaloffCoste2000Kac} proved the first bound, but it was exponential in $n$. The first polynomial-time bound was show  Jiang~\cite{Jiang2017KacSOPolynomial}, who showed the mixing time was $O(n^5 \log^2 n)$. This was improved shortly after by Pillai and Smith~\cite{PillaiSmith2018MixingKacWalk} to $O(n^4 \log n)$, and very recently, in a work concurrent to ours, to the near-optimal $O(n^2 \log n)$ by the same authors~\cite{pillai2026kacswalkrotationmatrices}. 

This paper is focused on the mixing time of the $k$-column walk. In the case $k=1$, the tight mixing bound of $O(n\log n)$ was shown by Pillai and Smith~\cite{pillaiKac17} after a sequence of works~\cite{DiaconisSaloffCoste2000Kac,Jiang2012TotalVariationKac}. To our knowledge, there are no strong bounds on the mixing time for $1 < k \ll n$ in any sense except for the relaxation time.

\paragraph{Our First Result.}
The first result in this paper fits squarely into the large literature on mixing times of Markov chains~\cite{levin2008markov}. It shows that the $k$-column walk mixes in $O(n\log(n) \cdot \max(k,\log(n)))$ steps uniformly in $k$ and $n$ in the Wasserstein metric, resolving a conjecture of \cite[pg.~1229]{Oliveira_2009}. 

\begin{theorem}[Informal]
\label{informal:kwise}
Let $D$ be the Riemannian distance on the Stiefel manifold $\mathrm{St}(n,k)$~\footnote{For $1 \le k \le n$, the Stiefel manifold is
$\mathrm{St}(n,k)=\{ X \in \mathbb{R}^{n\times k} : X^\top X = I_k \}$.} and let $W_D$ be Wasserstein distance with cost $D$. Let $P^T(X)$ denotes the probability distribution induced by running the $k$-column walk for $T$ steps starting with $X \in \mathrm{St}(n,k)$. Then, for every $X,Y\in\mathrm{St}(n,k)$, the Wasserstein distance $W_D\left(P^T(X),P^T(Y)\right)$ is at most $\epsilon$ as long as 
\[
T = \Omega\big(n(k+\log n) \cdot \log (D(X,Y)/\epsilon) \big)~.
\] 
\end{theorem}

When $k=n$, the $k$-column chain is the usual Kac walk on $\So{n}$ and the corresponding Wasserstein contraction estimate matches the one proved in \cite{Oliveira_2009}. Furthermore, for every $k$ and $n$, our bound is essentially optimal since $\Omega(nk)$ steps is absolutely necessary for convergence~\cite{Oliveira_2009}. We refer the reader to 
Theorem \ref{MainThm:kcol} for a formal statement, and to Section~\ref{SecContraction} for its proof.

\paragraph{Pseudo-Mixing and Our Main Result.}
While Kac's walk on $\So{n}$ does not mix in any traditional sense in $o(n^2)$ steps, it does seem to produce ``random enough'' matrices after much fewer, e.g. $n \cdot (\log n)^{O(1)}$ steps. 
A major consequence is that you get matrices that look random yet admit fast matrix operations; for example, a $T$-step Kac matrix admits matrix-vector multiplication in $O(T)$ time.

This phenomenon was first conjectured by Ailon and Chazelle~\cite{FastKacConj} in their groundbreaking work on fast Johnson-Lindenstrauss (JL) transforms. Recall that JL is a dimensionality reduction technique where, given a set $S$ of $N$ points in $\mathbb{R}^n$, one applies a Gaussian (or scaled Rademacher) matrix $A \in \mathbb{R}^{\ell \times n}$ to each $x\in S$ so that 
$$ (1-\epsilon) ||x-y||_2 \leq ||Ax-Ay||_2 \leq (1+\epsilon) ||x-y||_2$$
for all $x,y\in S$.
Remarkably, this can be achieved when $\ell$  is only  $\Theta(\log N / \epsilon^2)$. While a na{i}ve application of the JL lemma requires runtime $O(n\ell)$ for multiplication of $x$ by $A$, Ailon and Chazelle came up with several ways to reduce the runtime to $\tilde{O}(n)$, nearly the time required to just read the input $x$, by constructing the matrix $A$ cleverly. Curiously, they also propose an alternative construction, setting the matrix $A$ to be (the projection to the first $\ell$ rows of) a matrix resulting from Kac's walk after $T = O(n\log n + (\log N / \epsilon)^{O(1)})$  steps, and conjecture it to be ``at least as good as the one described in [their] paper, yet much more
elegant''. Indeed, matrix-vector multiplication with such a matrix can be accomplished in $O(T)$ time, much faster than $O(n\log N / \epsilon^2)$ guaranteed by vanilla JL. This conjecture was essentially resolved nearly 15 years later by Jain, Pillai, Sah, Sawhney and Smith~\cite{jain2020fastmemoryoptimaldimensionreduction}.

Yet, the phenomenon, that we will refer to as {\em pseudo-mixing}, appears to be much more general than that. In a recent work, Vaikuntanathan and Zamir~\cite{vaikuntanathan2025improvingalgorithmicefficiencyusing} conjectured that the Kac walk after $n\cdot (\log n)^{O(1)}$ steps is computationally indistinguishable from a ``uniformly random'' orthogonal matrix, namely from the Haar measure on $\So{n}$. That is, no polynomial-time algorithm can tell the two distributions of matrices apart. They relied on this conjecture to construct {\em trapdoored matrices}, distributions over $n \times n$ matrices which look (Haar-)random to any polynomial-time algorithm, yet admit near-linear time matrix-vector multiplication (and many other matrix operations as well). The computational indistinguishability here is crucial: they also show via a simple counting argument that truly (Haar-)random matrices (and ones that are TV- or Wasserstein-close to them) cannot possibly admit subquadratic time matrix-vector multiplication.

If the VZ conjecture is true, it says a far more general statement than the one predicted by Ailon and Chazelle. Not only is the Kac matrix (after a near-linear number of Kac steps) good enough for JL, it is good enough for {\em any polynomial-time computable} application. 

It is hopeless to prove the conjecture of \cite{vaikuntanathan2025improvingalgorithmicefficiencyusing} in its full generality as it would imply that $P \neq NP$, and even that cryptographic one-way functions exist unconditionally. Short of that, we could attempt to prove the conjecture against restricted classes of distinguishers. In this work, we consider the family of low-degree polynomials over $\mathbb{R}$, drawing from an illustrious line of work in the algorithmic statistics literature (see the thesis of Hopkins~\cite{hopkinsSOS} and the survey of Wein~\cite{wein2025computationalcomplexitystatisticsnew} for an extensive historical perspective.)

Our second, and main, result, is thus about Kac's walk on all of $\So{n}$ and says, roughly, that a certain class of test functions, namely (poly-logarithmically) low-degree polynomials, cannot distinguish between a sample from the Haar measure on $\So{n}$ and a sample obtained by running Kac's walk for $T \approx n \, \mathrm{polylog}(n)$ steps. This resolves an open problem posed by Vaikuntanathan and Zamir~\cite{vaikuntanathan2025improvingalgorithmicefficiencyusing}.

\begin{theorem}[Informal] \label{informal:MainThm} 
Polynomials over $\mathbb{R}$ of total degree at most $k=k(n)$ cannot distinguish between a sample from the Haar measure and a sample obtained by running Kac's walk for $T(n) = \tilde{O}(nk^2)$ steps.
\end{theorem}

We refer the reader to 
Theorem \ref{MainThm} for a formal statement, and to Section~\ref{SecMainThmProof} for its proof.

While this result may at first sight appear to contradict previous lower bounds on the mixing time of Kac's walk, such as the lower bound of $\Omega(n^{2})$ for mixing in Wasserstein due to \cite{Oliveira_2009}, it does not. The difference, of course, comes from the test functions being studied. Most of the literature on mixing times measures mixing using ``classical" test functions such as all bounded functions (total variation) and all Lipschitz functions (Wasserstein), while we follow \cite{hopkinsSOS,wein2025computationalcomplexitystatisticsnew} and use low-degree polynomials. Low-degree polynomials are not sufficient for all applications, but as we discuss in Section~\ref{SecApp}, they are good enough for many.

Following the low-degree heuristic championed by \cite{hopkinsSOS}, one could also informally interpret this as saying that it is computationally difficult for any ``natural'' polynomial-time algorithm to distinguish between Haar and Kac matrices. However, such a statement should be made with some care as recently demonstrated by \cite{buhai2025quasipolynomiallowdegreeconjecturefalse}; see also \cite{DBLP:conf/stoc/HsiehKKLMT26} for a more nuanced perspective of when predictions made by the low-degree heuristic can be turned into theorems.

\paragraph{Broader Perspective: Pseudo-Mixing as a General Phenomenon.}
While  Theorem \ref{MainThm} only applies to a single Markov chain, we believe our results are of broader interest. This is both because the Kac walk itself is of great interest, and because this is (to our knowledge) the first example of a phenomenon that would be of interest to the Markov chain Monte Carlo (MCMC) community:

\begin{itemize}

\item \textbf{MCMC:} The Markov chain studied here was the main worked example in the first major paper on MCMC \cite{hastings70}. The last sentence of that paper asked, essentially, the question we are trying to answer in Theorem \ref{MainThm}: {\em for what functions on $\So{n}$ do we expect this MCMC chain to be ``efficient"}? In light of previous work on this walk \cite{Oliveira_2009}, we believe the answer of this paper to be perhaps a bit surprising: the walk ``looks random" to certain computationally-bounded observers long before it ``looks random" according to most metrics on probability measures.

Beyond the historical interest in Kac's walk, we believe this result is qualitatively interesting to MCMC practitioners. It is well-known that mixing times can give conservative estimates for the convergence of specific functions of Markov chains \cite{Rabinovichfctn}, but most mathematically tractable examples appear to be ``toy" functions. Even when the phenomenon was empirically noticed for practical MCMC algorithms, it appeared to mostly require rather special functions. This is one of very few situations in which the usual mixing time is known to be very conservative for {\em natural and important} functions of a {\em natural} Markov chain. Having found one such example, we expect there to be many others. 

\item \textbf{Mathematical Techniques:} 
Our paper also includes nontrivial representation-theoretic arguments that are used to lower-bound the variances of polynomials on $\So{n}$ (see Section \ref{SecRepTheoryVariance}). These results may be of independent interest, as moment calculations on manifolds are generally difficult \cite{collins_weingarten_survey}. Such variance lower-bounds are necessary for one major approach to computational hardness \cite{wein2025computationalcomplexitystatisticsnew}, and could be used by researchers working on computational hardness against low-degree polynomials in constrained state spaces such as $\So{n}$.



\end{itemize}

\subsection{Applications} \label{SecApp}

\paragraph{Machine Learning and Fast Matrix Calculations.} \label{SecAppML}
As discussed in \cite{vaikuntanathan2025improvingalgorithmicefficiencyusing}, computational indistinguishability of $T$-step Kac walk for $T \ll n^2$ can be used in algorithmic applications. Informally, the heuristic is: if you can't ``easily'' tell the difference between a target distribution $\mu$ and an easier-to-sample distribution $\nu$, then surely it must be ``safe'' to replace $\mu$ by $\nu$ in any ``reasonable'' statistical or machine learning workflow. 


It is sometimes possible to upgrade this heuristic to a precise theorem. As an illustrative example, we show that Theorem \ref{MainThm} almost immediately implies that Kac's walk can, by itself, be used to give a fast Johnson-Lindenstrauss transform (see Theorem \ref{thm:kac-jl-transform}), providing a short proof of a conjecture of \cite{FastKacConj} who suggested this approach but did prove any formal guarantees. We note that our proof is \textit{inspired} by the computational hardness heuristic of \cite{vaikuntanathan2025improvingalgorithmicefficiencyusing}, but is not a heuristic argument. 

We also note that the algorithm in \cite{FastKacConj} is no longer the state-of-the-art fast Johnson-Lindenstrauss transform, and our theorem is no longer the state-of-the-art result. In the 20 years since \cite{FastKacConj}, there has been substantial work on developing slightly faster transforms \cite{AilonLiberty2013FastJL,DasguptaKumarSarlos2010SparseJL,KrahmerWard2011RIPJL,KaneNelson2014SparserJL,BambergerKrahmer2021OptimalFastJL,Fandina2023FastJLEvenFaster,HouenThorup2023FastHashing}, and the conjecture of \cite{FastKacConj}  was more-or-less proved in \cite{jain2020fastmemoryoptimaldimensionreduction}. We nonetheless view Theorem \ref{thm:kac-jl-transform} as interesting, since it shows how to obtain nontrivial results as a near-immediate consequence of a general bound on the computational mixing time.  


We suspect that this line of thinking is fruitful in a much broader set of applications and believe that it could be a modular source of speed-ups in various machine learning algorithms.
 

\paragraph{Cryptography.}
 A trapdoored matrix~\cite{sotirakitrap,vaikuntanathan2025improvingalgorithmicefficiencyusing,DBLP:conf/tcc/BravermanN25} consists of a public $n$-by-$n$ matrix $M$, as well as a private trapdoor function $P_{M}$ with the following properties, for some $0 < \varepsilon \ll 1$. 

\begin{enumerate}
    \item We have $P_{M}(v) = Mv$ for all vectors $v$.
    \item For any $v$, we can compute $P_{M}(v)$ in $O(n^{1 + \varepsilon})$ steps; however,
    \item For a computationally bounded adversary, $M$ is indistinguishable from a random matrix, and consequently, by a simple counting argument, a polynomial-time adversary cannot come up with a $o(n^2)$-size circuit that computes $v \mapsto Mv$. 
\end{enumerate}

%


Sotiraki~\cite{sotirakitrap} implicitly defined this notion and used it to construct an authentication protocol where the honest parties (the authenticators) run in $\tilde{O}(n)$ time, but no $o(n^2)$-time adversary can impersonate them. Her construction used  matrices that result from a {\em transvection walk} over a finite field $\mathbb{F}$, a process that is very similar to Kac's walk. A transvection walk~\cite{DiaconisSaloffCoste1996GeneratingSetsAbelian} is one where, in each step, one chooses rows $i$ and $j$ randomly and performs a random invertible linear transformation to them. This forms a Markov chain over $GL_n(\mathbb{F})$, similar to how Kac's walk works over $\So{n}$. She proved that each column of the resulting matrix, after a $\tilde{O}(n)$-step walk, is close to uniformly random over $\mathbb{F}^n$. These results have since then been improved, and much is known about the transvection walk by now~\cite{BenHamouPeres2018StratifiedHypercube,BenHamou2025MatrixTransvections,pillai2026mixingkcolumnstransvection}.

Vaikuntanathan and Zamir, and independently, Braverman and Newman~\cite{vaikuntanathan2025improvingalgorithmicefficiencyusing,DBLP:conf/tcc/BravermanN25} formally defined trapdoored matrices. 
\cite{vaikuntanathan2025improvingalgorithmicefficiencyusing} used a construction based on the Kac walk (conjecturing that the Kac matrix is computationally indistinguishable from a Haar matrix) and used it to speed up algorithms that use random matrices (as discussed above) as well as to show improved worst-case to average-case reductions for matrix multiplication over finite fields.  This latter application can be ported to the case of reals as well, as we briefly discuss below.  Given a black-box that multiplies two Haar-random matrices $A$ and $B$ in $\So{n}$ correctly with probability $\epsilon > 0$ (here, we ignore finite-precision issues), one can use it to compute $AB$ for {\em any} two matrices $A$ and $B$. Pick three Kac matrices $M_1, M_2$ and $M_3$. Feed the box with $A' = M_1 A M_2$ and $B' = M_2^{-1}BM_3$. The black-box returns a matrix $C'$, which if it did its job correctly, will be $M_1 AB M_3$. One can recover $AB$ from this by multiplying on the left by $M_1^{-1}$ and on the right by $M_3^{-1}$. Here, we use an additional property of Kac matrices: multiplication by a Kac matrix as well as by its inverse can be done in $\tilde{O}(n^2)$ time. One also crucially uses the fact that the correctness of the box, namely whether $C' = A'B'$ can be checked in $O(n^2)$ time using (a variation of) Frievald's checker~\cite{freivalds1979fast}. Note that this worst-case to average-case reduction is very different from the usual \cite{DBLP:conf/stoc/BlumLR90} type additive randomization over finite fields that, unfortunately, does not work over the reals. The consequence of our Theorem~\ref{MainThm} to this reduction is less clear, and we leave an in-depth exploration to future work. 

We also mention that variations of the cryptographic applications of \cite{sotirakitrap,DBLP:conf/tcc/BravermanN25} can also be implemented using Kac walk over $\So{n}$ as opposed to the transvection walk over a finite field, but leave the details to future work.

\subsection{Other Related Work}

The notion of a unitary $k$-design from quantum information asks for a distribution over unitary matrices in $SU(N)$ (where $N=2^n$, $n$ being the number of qubits in question, is the dimension of the underlying Hilbert space) whose $k$-th moments match those of a Haar-random matrix over $SU(N)$. This is very close to what we study in this paper. Indeed, there have been several recent works that construct $k$-designs, the closest to us being the construction based on random quantum circuits~\cite{tan2021approximateunitary3designstransvection,DBLP:conf/focs/ChenHHLMT25}. Random quantum circuits are unitaries defined by composing many local quantum gates, that is, gates that operate on $O(1)$ many qubits. However, the notion of locality here differs a great deal from the notion of locality in Kac's walk. A local unitary gate on $n$ qubits actually acts on many columns of the underlying unitary in $SU(2^n)$. On the other hand, a unitary on $SU(2^n)$ that acts on just two columns (such as in Kac's walk) is not a local gate in the sense that it has to act on many qubits.  Beyond this, the techniques in our work and \cite{DBLP:conf/focs/ChenHHLMT25} seem to be very different.


\paragraph{Acknowledgments and Statement of AI Use.}
ChatGPT 5.5 pro was used throughout the creation of this document. Many of the uses were related to saving time and making the proof structure cleaner, such as propagating changes to notation, or simplifying some of our earlier estimates (\textit{e.g.} Lemma \ref{lem:row_sum_secant_identity}), or drafting straightforward modifications of standard calculations (\textit{e.g.} Lemma \ref{lem:jl_haar_moment}). 
ChatGPT 5.5 Pro provided extensive guidance on the representation-theoretic literature that was most relevant for the calculations in Appendix~\ref{SecRepTheoryVariance}. We also thank Mark Sepanski 
for providing a more experienced set of eyes in reviewing the calculations in Appendix~\ref{SecRepTheoryVariance}. 
The final results are of course the responsibility of the authors.

\paragraph{Guide to the Paper.}
In Section~\ref{SecThmStatements}, we state the formal versions of our theorems and give an outline of the proofs. We prove the main error upper bound in Section \ref{SecContraction}, the matching lower bound in \ref{SecBoundsOnPoly}, and the main computational mixing bound in Section \ref{SecMainThmProof}. Section \ref{SecJL} gives an application to the analysis of fast Johnson-Lindenstrauss transforms. Most proof details are deferred to the appendices; the proofs in the main text are largely intended to show how those arguments can be tied together.

\section{Technical Overview} 
\label{SecThmStatements}

\subsection{Low-Degree Hardness and Pseudo-mixing} \label{SecCompMixingDefs}


Our definitions follow \cite[Definition 3.4]{wein2025computationalcomplexitystatisticsnew}.

\begin{defn} [Distinguishability With Respect to Fixed Sequences] \label{DefIndFunc}
Fix two sequences of distributions $\mu_{n}, \nu_{n}$ that take on values in the set $\mathcal{S}_{n}$, and let $f_{n}$ be a sequence of functions from $\mathcal{S}_{n}$ to $\mathbb{R}$.

We say the sequence $f_{n}$ \textit{weakly} distinguishes $\mu_{n}$ and $\nu_{n}$ if
\be 
\sqrt{\max(\Var_{\mu_{n}}(f_{n}),\Var_{\nu_{n}}(f_{n}))} = O(|\mathbb{E}_{\mu_{n}}[f_{n}] - \mathbb{E}_{\nu_{n}}[f_{n}] |),
\ee 
and \textit{strongly} distinguishes if the LHS is 
$o(|\mathbb{E}_{\mu_{n}}[f_{n}] - \mathbb{E}_{\nu_{n}}[f_{n}] |)$.
\end{defn}

Typically, we are interested in indistinguishability with respect to a large class of adversaries. The most important is the class of low-degree polynomials.

\begin{defn} [Indistinguishability With Respect to Low-Degree Polynomials] \label{DefInd}
Fix a monotone sequence $k(n)$. We say that $\mu_{n},\nu_{n}$ are weakly indistinguishable by degree-$k(n)$ polynomials if, for every sequence of polynomials $f_{n}$ with $deg(f_{n}) \leq k(n)$, $f_{n}$ does not weakly distinguish between $\mu_{n}, \nu_{n}$. The definition for strong indistinguishability is analogous.
\end{defn}

\noindent
This motivates the definitions of computational distance and computational mixing time.

\begin{defn} [Computational Distance]
Fix two distributions $\mu, \nu$ on a state space indexed by a size parameter $n$, and fix a degree parameter $d > 0$. We define
\be 
d_{\mathrm{comp}}(\mu,\nu) = d_{\mathrm{comp}}(\mu,\nu;n,d) \equiv  \sup_{f \in \mathcal{F}_{n,d}} | \mathbb{E}_{\mu}[f] - \mathbb{E}_{\nu}[f]|,
\ee 
where $\mathcal{F}_{n,d}$ is the collection of real polynomials in the ambient coordinates of degree at most $d$ satisfying
\be 
\max(\Var_{\mu}(f),\Var_{\nu}(f)) = 1.
\ee 
\end{defn}
\vnote{I made a major change in the def above. The paraemter $d$ refers to the degree of the polys and not the exponent of the logarithm as it was before... TODO: make sure this is propagated throughout the paper.}

\begin{defn} [Computational Mixing Time, or Pseudo-mixing Time]
Fix a sequence of Markov chains $\{X_{t}^{(n)}\}$ indexed by a size parameter $n$ and with unique stationary measures $\mu^{(n)}$. Fix a degree parameter $d > 0$ and a tolerance $\varepsilon > 0$. Define 
\be 
\tau_{\mathrm{comp}} = \min \{t > 0 \, : \, d_{\mathrm{comp}}(\mathcal{L}(X_{t}^{(n)}), \mu^{(n)}; n,d) \leq \varepsilon \}.
\ee 
where $\mathcal{L}(X_t^{(n)})$ is the probability law corresponding to the random variable $X_t^{(n)}$.
\end{defn}

\begin{remark}
A remark to experts on Markov chain theory: note that the definition of computational distance is very close to the definition of Wasserstein distance: both distances are defined as the largest difference in means across some collection of test functions, with only the test functions differing. Although this might make the notion of computational mixing or pseudo-mixing quite familiar to those with backgrounds in Markov chain theory, the computational distance does not have most of the familiar ``nice" properties of Wasserstein distances. For example, single-step contraction of the computational distance does not imply multi-step contraction.
\end{remark}

\subsection{Statements of Main Theorems}

Our result on the mixing time of the $k$-column walk is as follows.

\begin{thm}\label{MainThm:kcol}
There is a universal constant $C>0$ such that the following holds. For every $1\leq k<n$, let $P$ be the transition kernel of the $k$-column walk on $\mathrm{St}(n,k)$, let $D$ be the Riemannian distance on $\mathrm{St}(n,k)$, and let $W_D$ be Wasserstein distance with cost $D$. Then, for every $T\geq0$ and every pair of starting points $X_0,Y_0\in\mathrm{St}(n,k)$,
\[
W_D\left(P^T(X_0,\cdot),P^T(Y_0,\cdot)\right)
\leq
n^{-\left\lfloor T/\big(Cn(k+\log n)\log n\big)\right\rfloor}D(X_0,Y_0).
\]
\end{thm}

The proof is deferred to Section \ref{SecContraction}, where it is an immediate corollary of Theorem \ref{LemmaMainContractionEstimate}.

When $k=n$, if the chain is started in $\So{n}$, the $k$-column chain is the usual Kac walk on $\So{n}$; the corresponding Wasserstein contraction estimate is proved in \cite{Oliveira_2009}.

\medskip\noindent 

Let $\haarOdist_{n}$ denote the Haar measure on $\So{n}$. For $x\in\So{n}$, let $\kacdist_{n,T,x}$ denote the law of Kac's walk at time $T$ started from $x$. Our main theorem is:

\begin{thm}\label{MainThm}
Fix a sequence $k = k(n)$ and a sequence  $T = T(n)$ satisfying  $T(n) = \omega(nk \cdot (k+\log n) \cdot \log n)$, 
$\kacdist_{n,T,x}$ and $\haarOdist_{n}$ are weakly $k(n)$-indistinguishable in the sense of Definition \ref{DefInd}.
\end{thm}

The proof is deferred to Section \ref{SecMainThmProof} (which, in turn, heavily relies on Section \ref{SecBoundsOnPoly}).

\subsection{Proof Sketches}

We briefly sketch the proof of our main theorem, Theorem \ref{MainThm}. The proof is divided into two parts: an upper bound on the difference $|\kacdist_{n,T,x}(f)- \haarOdist_{n}(f)|$ and a lower bound on the variance $\Var_{\haarOdist_{n}}(f)$ for low-degree polynomials. 

To understand the upper bound, we begin by discussing some simple reductions. First, note that to control the error of any monomial in the entries of a matrix, it is sufficient to control the 1-Wasserstein distance between the matrices (see \textit{e.g.} the standard estimate in Lemma \ref{LemmaTelescopingSum}). Second, note that any monomial of degree $k$ is a function of variables that appear only in at most $k$ columns of Kac's walk. Thus, by linearity of expectations and the previous reduction, it is enough to check that the $k$-column walk mixes quickly in Wasserstein distance.

It is natural to try to mimic the existing sharp estimate on the Wasserstein mixing time of Kac's walk found in \cite{Oliveira_2009}. Unfortunately, there is an immediate problem. To see the problem, we briefly sketch \cite{Oliveira_2009}. The main calculation is a proof that two copies of Kac's walk that start at ``very close'' points $x,y$ can be coupled so that they ``contract'' at some rate $(1 - c)$. From that, a standard local-to-global argument implies that the walk mixes well enough for our purposes in time $c^{-1} \mathrm{polylog}(n)$. For the full walk on $\So{n},$ \cite{Oliveira_2009} shows that a contraction $c \approx n^{-2}$ holds for \textit{all} sufficiently nearby pairs, obtaining an optimal mixing time bound of roughly $n^{2} \mathrm{polylog}(n)$. If we could show an analogous bound of $c \approx \frac{1}{n(k+\log n)}$ for the $k$-column walk, we would be done.

Unfortunately, it does not seem to be possible to get such a bound uniformly in \textit{all} nearby points. Rather than finding a global contraction rate $c$, we find that the optimal constant depends on the row norms of the base points and can vary over roughly the range $0 \lesssim c \lesssim n^{-2}$. While the typical value of the contraction $c$ in the neighbourhood of a point sampled from Haar measure is roughly $c \approx \frac{1}{n(k+\log n)}$, this does not immediately give a one-step contraction estimate. Even worse, there are serious obstructions to going from any local contraction estimate to a global contraction estimates, since geodesic paths between points in high-contraction regions of $\So{n}$ may pass through low-contraction regions.\footnote{On $\mathrm{St}(2,1)$, the contraction rate near a point is proportional to 1 over the sup-norm. As a simple example of why contraction at points along a geodesic cannot be bounded by contraction at the endpoints, consider the points $X=(\frac{1}{\sqrt{2}},\frac{1}{\sqrt{2}})$ and $Y=(\frac{1}{\sqrt{2}},\frac{-1}{\sqrt{2}})$ on the unit circle $\mathrm{St}(2,1)$. These points have sup-norms satisfying $\|X\|_{row,\infty}^{2} = \|Y\|_{row,\infty}^{2} = \frac{1}{2}$, which is the smallest possible value on the unit circle. However, the geodesic must clearly pass through $Z=(1,0),$ which has sup-norm $\|Z\|_{row,\infty}^{2} = 1$, the largest possible value. } 

Our proof of the upper bound proceeds roughly as follows:

\begin{enumerate}
    \item We define an optimal coupling and use it to prove a local contraction estimate in Frobenius norm; see Lemma \ref{LemmaLocalCont}. The same coupling is deterministically non-expansive in Frobenius norm; see Lemma \ref{lem:one_step_stability}.
    \item We use burn-in estimates based on \cite{pillaiKac17} to show that, after a burn-in of order $n\log n$, individual  chains have row norms of order $\frac{k+\log n}{n}$ with high probability and thus are in high-contraction regions; see Section \ref{SecBurninEstimatesAll}. 
    \item We next consider the contraction of two copies of Kac's walk that are started very close to each other, over time periods of $T_{\mathrm{epoch}} \approx n\log(n)$ steps. Using Lemma \ref{lem:B1_multiplicative}, we combine the above burn-in estimates to show that these chains spend most of their time in high-contraction regions, the  contraction estimate \eqref{IneqLocalContractFrob} in these high-contraction regions, and the non-expansion bound \eqref{EqLocalEpochNonExpansion} outside of these regions, to obtain the multi-step contraction estimate \eqref{eq:epoch_boundary_contraction_final_conclusion} in Lemma \ref{lem:epoch_contraction_v2_refactored}.
    \item We apply a standard path-coupling result, Lemma \ref{lem:local_to_global_wasserstein}, to go from these local contraction estimates to global contraction estimates..
\end{enumerate}

The lower bound on the variance is a rather involved representation-theoretic argument, contained in Section \ref{SecRepTheoryVariance}. Compared to standard variance lower bounds in \cite{wein2025computationalcomplexitystatisticsnew}, the main difficulty here is illustrated by the fact that some nontrivial polynomials in the matrix entries, such as $\sum_{i,j} X_{i,j}^{2}$, have 0 variance under the Haar measure. For this reason, it is impossible to get any nontrivial lower bound purely in terms of the sizes of the coefficients, and any calculation must deal with very nontrivial correlation after ``projecting away" from the space of zero-variance polynomials. The majority of our argument consists of projecting to a space of polynomials with nonzero variance, then carrying out exact calculations using the well-understood representations of $\So{n}$. These are combined with a general projection inequality, Theorem \ref{thm:compact-group-variance}, that might be useful for other statistical problems on manifolds.

Finally, in terms of applications, we also show how our main results can be applied to get nontrivial bounds for a real algorithm: the fast Johnson-Lindenstrauss transform (see Theorem \ref{thm:kac-jl-transform}). The basic proof strategy is very similar to the elementary proof of the famous Johnson-Lindenstrauss theorem in  \cite{DasguptaGupta2003ElementaryProofJL}: we merely replace the bounds on exponential moments in  \cite{DasguptaGupta2003ElementaryProofJL} by bounds on moments of polynomials with low degree and propagate these changes. 


\section{Contraction in Columns of Kac's Walk} \label{SecContraction}

\vnote{Starting here, should carefully check that we say ``indistinguishable by low-degree polys" rather than "computationally indistinguishable".} 

The main ingredient in the proof of Theorem \ref{MainThm} is a proof that certain functions of Kac's walk mix much more quickly than the walk itself. In this section, we make this precise and prove the following mixing bound (with some notation made precise later in this section):

\begin{thm}[Global Multi-Step Contraction]
\label{LemmaMainContractionEstimate}
Fix $n \in \mathbb{N}$ and $k \in \mathbb{N}$ with $1\leq k < n$. Let $P$ be the one-step transition kernel of the $k$-column walk on $\mathrm{St}(n,k)$, let $D$ be the Riemannian distance on $\mathrm{St}(n,k)$, and let $W_D$ be Wasserstein distance with cost $D$. Then there exists a constant $C > 0$ such that, for every $T\geq0$ and every $X_0,Y_0\in\mathrm{St}(n,k)$,
\[
W_D\left(P^T(X_0,\cdot),P^T(Y_0,\cdot)\right)
\;\le\;
n^{-\left\lfloor \frac{T}{C n (k+\log(n))\log(n)} \right\rfloor}\,D(X_0,Y_0).
\]
\end{thm}

\noindent
The case $k=n$ is the full Kac walk on $\So{n}$ and is already covered by \cite{Oliveira_2009}, so we restrict the statement and proof to $1\leq k<n$. The proof is completed at the end of this section, using the epoch estimate proved in Section \ref{SecBurninEstimatesAll}.

The proof relies on somewhat lengthy calculations on $\mathrm{St}(n,k)$, so we introduce notation before returning to the Markov chain itself.

\subsection{Notation and Basic Facts Related to Geometry}

We write
\[
\So{n}=\{Q\in\mathbb{R}^{n\times n}:Q^{\top}Q=I_n,\ \det(Q)=1\}.
\]
We will typically equip $\So{n}$ with the Riemannian structure induced by the Hilbert-Schmidt inner product
\[
\langle A,B\rangle \equiv \mathrm{Tr}(A^\top B),
\qquad A,B\in\mathbb{R}^{n\times n}.
\]

\noindent
The tangent space at the identity is the Lie algebra of skew-symmetric matrices
\[
\mathfrak{so}(n) = \{ A \in \mathbb{R}^{n\times n} : A^\top = -A \},
\]
and the tangent space at a general point $Q\in \So{n}$ is
\[
T_Q \So{n} = \{ AQ : A \in \mathfrak{so}(n) \}.
\]
For $1 \le i < j \le n$, define the skew-symmetric matrix
\[
a_{ij} = \frac{1}{\sqrt{2}}(e_i e_j^\top - e_j e_i^\top) \in \mathfrak{so}(n),
\]
where $\{e_1,\dots,e_n\}$ is the standard basis of $\mathbb{R}^n$. We note that $\{a_{ij}\}_{1 \leq i < j \leq n}$ is an orthonormal basis of $\mathfrak{so}(n)$. For convenience, we define $a_{ji} = a_{ij}$ for $i > j$, noting that adding these to the earlier basis would of course prevent it from being a basis.

\noindent
For $1 \le k \le n$, the Stiefel manifold is
\[
\mathrm{St}(n,k)
=
\{ X \in \mathbb{R}^{n\times k} : X^\top X = I_k \}.
\]

\noindent
The tangent space at $X\in\mathrm{St}(n,k)$ is
\[
T_X \mathrm{St}(n,k)
=
\{ H \in \mathbb{R}^{n\times k} : X^\top H + H^\top X = 0 \}.
\]

For two matrices $x,y \in \mathrm{St}(n,k)$, we denote by $\| x - y \|$ the usual Frobenius norm. This agrees with the Hilbert-Schmidt norm on $\So{n}$ when $n=k$. Let $D$ be the Riemannian metric on $\mathrm{St}(n,k)$ obtained by viewing $\mathrm{St}(n,k)$ as an embedded submanifold of $\mathbb{R}^{nk}$.

If $X \in \So{n}$, denote by $\mathcal{P}_{k}(X) \in \mathrm{St}(n,k)$ the first $k$ columns of $X$. We note that, if $\{X_{t}\}$ is a copy of Kac's walk, then $\{\mathcal{P}_{k}(X_{t})\}$ is also a Markov chain. We call this chain the $k$-column walk.

Finally, for $i \in [n]$ and $X \in \mathrm{St}(n,k)$, we will let $r_{i}(X) = (X[i,1],\ldots,X[i,k])$ be the $i$-th row of $X$. Define $\|X\|_{row,\infty}^{2} = \max_{i} \|r_{i}(X)\|^{2}$.

In addition to this notation, we record some basic relationships between the main metrics of interest. For all $x,y\in\mathrm{St}(n,k)$,
\be \label{IneqLocalEquivFrobRiem}
\|x-y\| \leq D(x,y).
\ee
In the other direction, we claim that there exist universal constants $C_D<\infty$ and $\rho_D>0$ not depending on $n$ or $k$ such that, whenever $\|x-y\|\leq \rho_D$,
\be \label{IneqLocalUpperFrobRiem}
D(x,y) \leq \|x-y\| + C_D \|x-y\|^{2}.
\ee
We believe that inequality \eqref{IneqLocalUpperFrobRiem} is standard, though we don't know a reference. See Lemma~\ref{lem:stiefel_metric_comparison} for a complete proof.

\subsection{Local Contraction and Averaging}\label{SecCoupAvg}

The first important step in our argument is a local contraction estimate in Frobenius norm. The coupling used in the proof below optimizes over the two rows updated by one step of Kac's walk. For $k=n$ the angle in \eqref{EqAlphaOptimal} agrees with the angle in \cite{Oliveira_2009} up to first-order terms, but in general it can be substantially different.

\begin{lemma}[Local Contraction Bound] \label{LemmaLocalCont}
There is a universal constant $c_{\mathrm{loc}}>0$ such that the following holds. Let $X,Y\in \mathrm{St}(n,k)$ satisfy $\|X-Y\|\leq 1$, and let
\be
M = \max\left(\max_{a \in [n]} \| r_{a}(X)\|^{2}, \max_{a \in [n]} \|r_{a}(Y) \|^{2}\right).
\ee
Then there is a coupling $(X',Y')$ of one step of the $k$-column walk from $X$ and $Y$ such that
\be \label{IneqLocalContractConc}
\mathbb{E}[\|X'-Y' \|^{2}\mid X,Y]
\leq \left(1 - \frac{c_{\mathrm{loc}}}{Mn^{2}} \right) \|X-Y\|^{2}.
\ee
After decreasing $c_{\mathrm{loc}}$ by a universal factor, the same coupling also satisfies the Frobenius-distance contraction
\be \label{IneqLocalContractFrob}
\mathbb{E}[\|X'-Y' \|\mid X,Y]
\leq \left(1 - \frac{c_{\mathrm{loc}}}{Mn^{2}} \right) \|X-Y\|.
\ee
\end{lemma}

\begin{proof}
We first construct the one-step coupling. Sample $i,j$ uniformly from $\{i,j\in[n]:1\leq i<j\leq n\}$ and sample $\theta$ uniformly from $[0,2\pi)$. For an angle shift $\alpha=\alpha(X,Y,i,j)$, set
\be\label{EqCoupling}
X' = R_{ij}(\theta) X, \qquad
Y' = R_{ij}(\theta + \alpha) Y.
\ee
As long as $\alpha$ is a deterministic function of $X,Y,i,j$, \eqref{EqCoupling} is a valid coupling of the two one-step marginals: conditionally on $X,Y,i,j$, the shifted angle $\theta+\alpha$ is uniform modulo $2\pi$.

We now choose $\alpha$. For the selected pair $(i,j)$, write
\[
B_X^{ij}=\begin{pmatrix} r_i(X) \\ r_j(X) \end{pmatrix},
\qquad
B_Y^{ij}=\begin{pmatrix} r_i(Y) \\ r_j(Y) \end{pmatrix}
\in \mathbb{R}^{2\times k},
\qquad
J=\begin{pmatrix}0&1\\-1&0\end{pmatrix}.
\]
Define
\be\label{EqPQDef}
p_{ij}=\langle B_X^{ij},B_Y^{ij}\rangle_F,
\qquad
q_{ij}=\langle B_X^{ij},J B_Y^{ij}\rangle_F.
\ee
If $p_{ij}^{2}+q_{ij}^{2}>0$, define $\widehat\alpha_{ij}(X,Y)\in(-\pi,\pi]$ by
\be\label{EqAlphaOptimal}
\cos\widehat\alpha_{ij}(X,Y)=\frac{p_{ij}}{\sqrt{p_{ij}^{2}+q_{ij}^{2}}},
\qquad
\sin\widehat\alpha_{ij}(X,Y)=\frac{q_{ij}}{\sqrt{p_{ij}^{2}+q_{ij}^{2}}}.
\ee
If $p_{ij}=q_{ij}=0$, set $\widehat\alpha_{ij}(X,Y)=0$. In \eqref{EqCoupling}, we take $\alpha=\widehat\alpha_{ij}(X,Y)$.

This choice maximizes the two-row overlap
\be\label{EqTwoRowOverlap}
\langle B_X^{ij}, R^{(2)}_{\alpha} B_Y^{ij}\rangle_F
= p_{ij}\cos\alpha+q_{ij}\sin\alpha,
\ee
where $R^{(2)}_{\alpha}=\begin{pmatrix}\cos\alpha&\sin\alpha\\-\sin\alpha&\cos\alpha\end{pmatrix}$. For every angle $\alpha$, applying the common inverse rotation $R_{ij}(-\theta)$ to both points in \eqref{EqCoupling} gives
\be\label{EqTwoRowDistanceFormula}
\|X-R_{ij}(\alpha)Y\|^{2}
=\|X-Y\|^{2}-2\Big(p_{ij}\cos\alpha+q_{ij}\sin\alpha-p_{ij}\Big).
\ee
Combining \eqref{EqAlphaOptimal} with \eqref{EqTwoRowDistanceFormula} gives
\be\label{EqExactDecrease}
\|X-Y\|^2-\|X'-Y'\|^2
=2\left(\sqrt{p_{ij}^{2}+q_{ij}^{2}}-p_{ij}\right)
\equiv \Delta_{ij}.
\ee

If
\[
M\geq \max\bigl(\|X\|_{row,\infty}^{2},\|Y\|_{row,\infty}^{2}\bigr),
\]
then $|p_{ij}|,|q_{ij}|\leq \|B_X^{ij}\|_F\|B_Y^{ij}\|_F\leq 2M$. We claim that there exists a universal constant $c_\Delta>0$ such that
\be\label{EqDeltaLowerByq}
\Delta_{ij}\geq \frac{c_\Delta}{M}q_{ij}^{2}.
\ee
If $q_{ij}=0$ the claim is trivial. Otherwise, rewriting Equation \eqref{EqExactDecrease} gives
\[
\Delta_{ij}
=2\frac{q_{ij}^{2}}{\sqrt{p_{ij}^{2}+q_{ij}^{2}}+p_{ij}},
\]
and the denominator is at most $6M$.

Let $K=Y-X$. By \eqref{EqExactDecrease} and \eqref{EqDeltaLowerByq},
\[
\mathbb{E}[\|X'-Y'\|^{2}\mid X,Y]
\leq \|K\|^{2}
-\frac{c_\Delta}{M\binom{n}{2}}\sum_{i<j}q_{ij}^{2}.
\]
Since $B_Y^{ij}=B_X^{ij}+B_K^{ij}$ and $\langle B_X^{ij},JB_X^{ij}\rangle_F=0$, we have
\[
q_{ij}=r_i(X)\cdot r_j(K)-r_j(X)\cdot r_i(K).
\]
Lemma \ref{lem:row_sum_secant_identity} gives, for $\|K\|\leq 1$,
\[
\sum_{i<j}q_{ij}^{2}\geq \frac34\|K\|^{2}.
\]
Combining \eqref{EqDeltaLowerByq} with Lemma \ref{lem:row_sum_secant_identity}, and decreasing the universal constant if necessary, gives \eqref{IneqLocalContractConc}. Jensen's inequality and the standard bound $\sqrt{1-u}\leq 1-u/2$ then give \eqref{IneqLocalContractFrob}, again after decreasing $c_{\mathrm{loc}}$ by a universal factor.
\end{proof}

Because the term $M$ appearing in Lemma \ref{LemmaLocalCont} can be as large as 1, the contraction bound in Lemma \ref{LemmaLocalCont} can be of the order $(1 - \Theta(n^{-2}))$ for the worst-case starting points $X,Y$. This poor contraction is real, not merely an artifact of the proof strategy. We get around this problem by showing that $M$ is typically of order $\frac{k+\log n}{n}$ after a burn-in of order $\Omega(n \log n)$. The details of this calculation are given in Appendix \ref{SecBurninEstimatesAll}. The main estimate we need from that appendix is the following consequence of Lemma \ref{lem:epoch_contraction_v2_refactored}.

\begin{corollary} \label{CorWhatWeNeedD}
Let $P$ be the one-step transition kernel of the $k$-column walk and let $W_D$ denote Wasserstein distance with cost $D$. There are universal constants $C,\rho_{\mathrm{loc}}>0$ with the following property. For all sufficiently large $n$, if $x,y\in \mathrm{St}(n,k)$ satisfy $D(x,y)\leq \rho_{\mathrm{loc}}$, then
\be\label{EqEpochLocalForPathCoupling}
W_D\left(P^{T_{\mathrm{epoch}}}(x,\cdot),P^{T_{\mathrm{epoch}}}(y,\cdot)\right)
\leq n^{-1}D(x,y),
\ee
where $T_{\mathrm{epoch}}$ is as in Definition \ref{def:C4_epoch} and satisfies
\[
T_{\mathrm{epoch}}\leq Cn(k+\log n)\log n.
\]
\end{corollary}

\subsection{From Local to Global Contraction}\label{SecLocalToGlobalContraction}

We use the standard local-to-global criterion for Wasserstein contraction on a geodesic metric space. The following lemma follows immediately from \cite[Proposition~3.1]{Paulin2016MixingConcentrationRicci} and \cite[Propositions~19 and~20]{Ollivier_2009}. It is not substantially different from the original path-coupling argument of \cite{BubleyDyer1997PathCoupling}, except for technical difficulties related to being on a continuous state space. 

\begin{lemma}[Local-to-global Wasserstein contraction]\label{lem:local_to_global_wasserstein}
Let $(E,D)$ be a compact geodesic metric space and let $K$ be a Markov kernel on $E$. Let $W_D$ denote Wasserstein distance with cost $D$. Suppose that, for some $\rho>0$ and $\lambda\in[0,1)$,
\be\label{EqPathCouplingHypothesis}
W_D(K(x,\cdot),K(y,\cdot))\leq \lambda D(x,y)
\qquad \text{whenever }D(x,y)\leq \rho.
\ee
Then, for all $x,y\in E$,
\be\label{EqPathCouplingPointMassConclusion}
W_D(K(x,\cdot),K(y,\cdot))\leq \lambda D(x,y).
\ee
Consequently, for all probability measures $\mu,\nu$ on $E$,
\be\label{EqPathCouplingMeasureConclusion}
W_D(\mu K,\nu K)\leq \lambda W_D(\mu,\nu).
\ee
\end{lemma}

We are now ready to prove Theorem \ref{LemmaMainContractionEstimate}.

\begin{proof}[Proof of Theorem \ref{LemmaMainContractionEstimate}]
Let $P$ be the one-step transition kernel of the $k$-column walk and set $K=P^{T_{\mathrm{epoch}}}$. By Corollary \ref{CorWhatWeNeedD}, for all $x,y\in \mathrm{St}(n,k)$ with $D(x,y)\leq \rho_{\mathrm{loc}}$,
\be\label{EqEpochLocalForPathCouplingInProof}
W_D(K(x,\cdot),K(y,\cdot))\leq n^{-1}D(x,y).
\ee
For $1\leq k<n$, the space $\mathrm{St}(n,k)$ with the Riemannian distance $D$ is a compact geodesic metric space.\footnote{The only non-obvious part of this claim is the fact that any two points are joined by a minimizing geodesic. This is the content of the Hopf--Rinow theorem; see, for example,
\cite[Theorem~6.13]{lee2018riemannian}.} Applying Lemma \ref{lem:local_to_global_wasserstein} to \eqref{EqEpochLocalForPathCouplingInProof}, with $\rho=\rho_{\mathrm{loc}}$ and $\lambda=n^{-1}$, gives
\be\label{EqEpochGlobalForPathCoupling}
W_D(K(x,\cdot),K(y,\cdot))\leq n^{-1}D(x,y)
\ee
for all $x,y\in \mathrm{St}(n,k)$. The measure-level conclusion \eqref{EqPathCouplingMeasureConclusion} gives the same contraction after each subsequent epoch.

Let $r=\left\lfloor T/T_{\mathrm{epoch}}\right\rfloor$ and write $T=rT_{\mathrm{epoch}}+s$ with $0\leq s<T_{\mathrm{epoch}}$. Iterating \eqref{EqPathCouplingMeasureConclusion} with $K=P^{T_{\mathrm{epoch}}}$ and $\lambda=n^{-1}$ gives
\[
W_D\left(P^{rT_{\mathrm{epoch}}}(X_0,\cdot),P^{rT_{\mathrm{epoch}}}(Y_0,\cdot)\right)
\leq n^{-r}D(X_0,Y_0).
\]
Since our coupling is non-expansive (see Lemma \ref{lem:one_step_stability}), the last $s$ steps cannot increase the distance between the chains, and therefore cannot increase the expected distance. This completes the proof.
\end{proof}

\section{Bounds on Moments of Polynomials} \label{SecBoundsOnPoly}

We give bounds on the moments of low-degree polynomials under both (a) Haar measure and (b) the measure of Kac's walk on $\So{n}$ after $T$ steps.

\subsection{Upper Bounds on Differences of Means}

We use the bounds in Section \ref{SecContraction}. Recall $\kacdist_{n,T,x}$ is the distribution of a copy of Kac's walk on $\So{n}$ started at $X_{0}=x$ after $T$ steps, and $\haarOdist_{n}$ is the Haar measure on $\So{n}$. Our main estimate is:

\begin{theorem}\label{LemmaDiffUpperBoundMonomial}
For $n\in \mathbb{N}$, fix $1\leq k<n$ and $T\geq0$. Then fix a function $f= f_{n}$ with $f \, : \: \So{n} \mapsto \mathbb{R}$ of the form 
\[
f(X) = \prod_{r=1}^{k} X[i_{r},j_{r}].
\]

Let $C > 0$ be the constant appearing in the statement of Theorem \ref{LemmaMainContractionEstimate} and let 
\be \label{eq:omega_def}
\omega_{n} = C_{\mathrm{diam}}\,k\, n^{- \left\lfloor \frac{T}{C n (k+\log(n))\log(n)} \right\rfloor}
\ee 
be the upper bound obtained in that theorem.\footnote{Here we replace $D(X,Y)$ with the upper bound $C_{\mathrm{diam}} \, k$, where $C_{\mathrm{diam}}$ is a universal constant chosen so that the diameter of $\mathrm{St}(n,k)$ with respect to the Riemannian distance $D$ is at most $C_{\mathrm{diam}}k$.}

There exists some universal constant $N_{0}$ so that, for all $n > N_{0}$ and all $x \in \So{n}$,
\be 
| \kacdist_{n,T,x}(f) - \haarOdist_{n}(f) | \leq 2\sqrt{2k\,\omega_{n}}.
\ee 
\end{theorem}

\begin{proof}

For convenience, define the sequences 
\be \label{EqDefEpsDelt}
\varepsilon_{n} = \sqrt{\frac{2\omega_{n}}{k}}, \qquad \delta_{n} = \frac{\omega_{n}}{\varepsilon_{n}}.
\ee 

We note that $f(x)$ depends on only the entries of $x$ that appear in columns in the set $\{j_{r}\}_{r=1}^{k}$, and this set clearly has at most $k$ entries. Without loss of generality, assume that $\{j_{r}\}_{r=1}^{k} \subset [k]$. 

The proof of Theorem \ref{LemmaMainContractionEstimate} also gives the same contraction for any starting distribution, not merely point masses. Applying this to the point mass at $\mathcal P_k(x)$ and to the stationary $k$-column Haar law gives, by the definition of $\omega_n$ in \eqref{eq:omega_def},
\[
W_D\left(\mathcal{L}(\mathcal{P}_{k}(X_T)),\mathcal{L}(\mathcal{P}_{k}(Y))\right)\leq \omega_n,
\]
where $Y \sim \haarOdist_{n}$. Since $\mathrm{St}(n,k)$ is compact, this Wasserstein distance is realized by an optimal transport plan. Thus we can couple $\mathcal{P}_{k}(X_T)$ to $\mathcal{P}_{k}(Y)$ so that $\mathbb{E}[D(\mathcal{P}_{k}(X_{T}),\mathcal{P}_{k}(Y))]\le \omega_{n}$. Hence Markov's inequality gives
\be\label{IneqApplLemmaNormBds}
\mathbb{P}[D(\mathcal{P}_{k}(X_{T}),\mathcal{P}_{k}(Y)) > \varepsilon_{n}] \leq \delta_{n}.
\ee 

Next, we obtain entrywise bounds. Denote by $\| \cdot \|_{\infty}$ the entrywise supnorm. By \eqref{IneqLocalEquivFrobRiem}, we have $\|\mathcal{P}_{k}(X_{T})-\mathcal{P}_{k}(Y)\|\le D(\mathcal{P}_{k}(X_{T}),\mathcal{P}_{k}(Y))$, and the standard norm inequality $\|A\|_{\infty}\le \|A\|$ implies
\be \label{BasicSupBoundColumn}
\mathbb{P}[\| \mathcal{P}_{k}(X_{T}) - \mathcal{P}_{k}(Y)\|_{\infty}  > \varepsilon_{n}] \leq \delta_{n}.
\ee

On the event $\{ \| \mathcal{P}_{k}(X_{T}) - \mathcal{P}_{k}(Y)\|_{\infty}  \leq \varepsilon_{n}\}$, we have\footnote{We are using a standard telescoping-sum argument here. See Lemma \ref{LemmaTelescopingSum} in the appendix for a detailed calculation of the telescoping-sum identity; we then apply the triangle inequality to the result.}

\be 
|f(X_{T}) - f(Y)| \leq k\,\varepsilon_{n}.
\ee 

Combining this with Inequality \eqref{BasicSupBoundColumn}, and using the deterministic bound $|f(x)| \leq 1$, we conclude 
\be 
|\mathbb{E}[f(X_{T})] - \mathbb{E}[f(Y)]| \leq k\,\varepsilon_{n} + 2\delta_{n} = 2\sqrt{2k\,\omega_{n}}.
\ee 
\end{proof}

\subsection{Lower Bounds on Variances}
For a polynomial $f(x) = \sum_{u} c_{u} x^{u}$, define the $L^{2}$ norm of the coefficients by $\|f\|_{coeff,p} = \left(\sum_{u} c_{u}^{p}\right)^{\frac{1}{p}}$. The following is an immediate consequence of Corollary \ref{cor:inhomogeneous-case}:

\begin{theorem} \label{ThmVarLowerBound}
Fix a degree-$k$ polynomial $f \, : \: \So{n} \mapsto \mathbb{R}$ with mean $\haarOdist_{n}(f) = 0$ under the Haar measure. Then there exists a polynomial $f^{*} \, : \: \So{n} \mapsto \mathbb{R}$ such that, for all $x \in \So{n}$,
\be 
f(x) = f^{*}(x),
\ee 
and 
\be 
\Var_{\haarOdist_{n}}(f^{*}) \geq \frac{\|f^{*}\|_{coeff,2}^{2}}{(k+1)n^{k}}.
\ee 
\end{theorem}

Here $f^{*}$ is obtained from the projection $T^\sharp$ defined in Equation \eqref{EqStandardizingProjection}, after choosing a degree-at-most-$k$ representative of $f$. As it is somewhat unusual, we explain why the ``reduced" form $f^{*}$ is required. While two polynomials with distinct coefficients can never agree over all of $\mathbb{R}^{d}$, two such polynomials \textit{can} agree on $\So{n}$. A typical example is the polynomial $f(x) = \sum_{i,j=1}^{n} x[i,j]^{2}$, which is equal to the constant polynomial $g(x) \equiv n$ over $\So{n}$ (and in particular $f$ has variance 0 with respect to Haar measure). For this reason, it is impossible to obtain a variance lower bound in terms of the norm of the coefficients for ``unreduced" polynomials, but such a bound turns out to hold for these ``reduced" polynomials. 

\section{Proof of Our Main Theorem (Theorem \ref{MainThm})} \label{SecMainThmProof}

Let us first restate Theorem~\ref{MainThm} here for convenience.

\begin{theoremnonum}[Theorem~\ref{MainThm}, copied]\label{MainThm:copy}
Fix a sequence $k = k(n)$ and a sequence  $T = T(n)$ satisfying  $T(n) = \omega(nk \cdot (k+\log n) \cdot \log n)$, 
$\kacdist_{n,T,x}$ and $\haarOdist_{n}$ are weakly $k(n)$-indistinguishable in the sense of Definition \ref{DefInd}.
\end{theoremnonum}

Fix a polynomial $f$; since replacing $f$ by the ``reduced" polynomial $f^{*}$ appearing in the statement of Theorem \ref{ThmVarLowerBound} doesn't change its value on $\So{n}$, we can assume without loss of generality that $f$ is reduced in this sense (see Equation \eqref{EqStandardizingProjection} for an explicit formula for this reduction). Similarly, by simply dividing through the entire polynomial by a constant, we can assume without loss of generality that $\|f^{*}\|_{coeff,2}^{2} = 1$.

By Theorem \ref{ThmVarLowerBound}, 
\be \label{FinalFinalLB}
\Var_{\haarOdist_{n}}(f) \geq \frac{1}{(k+1)n^{k}}.
\ee 

Next, we compare this to the differences. By Cauchy--Schwarz and the graded coefficient representation used in Corollary \ref{cor:inhomogeneous-case}, $\|f \|_{coeff,1} \leq \sqrt{k+1}\,n^{k}$. By Theorem \ref{LemmaDiffUpperBoundMonomial} and linearity of expectations, there exists a universal constant $C > 0$ so that
\be \label{FinalFinalUB}
|\kacdist_{n,T,x}(f)- \haarOdist_{n}(f)|^{2} \leq 8 C_{\mathrm{diam}}\,(k+1)k^{2} n^{2k}\,
n^{- \left\lfloor \frac{T}{C n (k+\log(n))\log(n)} \right\rfloor}
\ee 

Comparing Inequalities \eqref{FinalFinalLB} and \eqref{FinalFinalUB}, we see that
\be \label{FinalFinalratio}
\frac{|\kacdist_{n,T,x}(f)- \haarOdist_{n}(f)|^{2}}{\Var_{\haarOdist_{n}}(f)} \leq 8C_{\mathrm{diam}}\, (k+1)^2k^2n^{3k}\,
n^{- \left\lfloor \frac{T}{C n (k+\log(n))\log(n)} 
\right\rfloor}
\ee 
The polynomial prefactor contributes only $(k+1)^{2}k^{2}n^{3k}$, while the contraction factor contributes $n^{-r}$ with
\[
r=\left\lfloor \frac{T}{C n (k+\log n)\log n}\right\rfloor .
\]
meaning that the quantity on the LHS in equation~\ref{FinalFinalratio} tends to $0$ superpolynomially fast if $T = \omega(nk \cdot (k+\log n) \cdot \log n)$.

\section{An Application to a Fast Johnson-Lindenstrauss Transform} \label{SecJL}

In this section, we show that the matrix obtained from running Kac's walk gives a Johnson--Lindenstrauss transform at the usual target dimension, provided the walk is run long enough to control the low moments used below. This confirms the Johnson--Lindenstrauss conjecture from \cite{FastKacConj}, up to the extra logarithmic factor in the running time appearing in Theorem \ref{thm:kac-jl-transform}.

Let $\Pi_\ell:\mathbb{R}^n\to\mathbb{R}^\ell$ be the coordinate projection onto the first $\ell$ coordinates. For $G\in\So{n}$, define
\be \label{EqJLAGDef}
A_G=\sqrt{\frac n\ell}\,\Pi_\ell G.
\ee
For a unit vector $v\in \mathbb{S}^{n-1}$, define the squared-norm distortion
\be \label{EqJLZvDef}
Z_v(G)=\|A_Gv\|_2^2-1=\frac n\ell\sum_{a=1}^{\ell}\langle e_a,Gv\rangle^2-1.
\ee
Thus $Z_v(G)$ is a degree-$2$ polynomial in the entries of $G$. For a finite set $\mathcal X=\{x_1,\ldots,x_N\}\subset\mathbb{R}^n$, define
\be \label{EqJLVXDef}
\mathcal V_{\mathcal X}=\left\{\frac{x_i-x_j}{\|x_i-x_j\|_2}:1\leq i<j\leq N,\ x_i\neq x_j\right\}.
\ee
For $v=(x_i-x_j)/\|x_i-x_j\|_2$, the identity \eqref{EqJLZvDef} gives
\be \label{EqJLPairDistortion}
Z_v(G)=\frac{\|A_G(x_i-x_j)\|_2^2}{\|x_i-x_j\|_2^2}-1.
\ee
Consequently, controlling $|Z_v(G)|$ for every $v\in\mathcal V_{\mathcal X}$ is exactly the Johnson--Lindenstrauss condition for the point set $\mathcal X$.

The application of our estimates to Kac's walk to the JL transform is the following.

\begin{theorem}[Kac's walk applied to the Johnson--Lindenstrauss transform] \label{thm:kac-jl-transform}
There is a universal constant $C_{\mathrm{JL}}<\infty$ such that the following holds. Let $n\geq 2$, let $x\in\So{n}$, and let $M\sim\kacdist_{n,T,x}$ be the endpoint of $T$ steps of Kac's walk started from $x$. Let $\mathcal X\subset\mathbb{R}^n$ have $N\geq2$ points, and let $0<\varepsilon,\delta<1/2$. Set
\be \label{EqJLKacRChoice}
r=\left\lceil \log\frac{N}{\delta}\right\rceil,
\ee
and assume $4r<n$. Let $\ell$ be an integer satisfying
\be \label{EqJLKacEllChoice}
C_{\mathrm{JL}}\varepsilon^{-2}r\leq \ell<n.
\ee

If
\be \label{EqJLKacTimeAssumption}
T\geq
C_{\mathrm{JL}}\,n(4r+\log n)\left(r\log\frac{n}{\varepsilon}+\log\frac{N}{\delta}\right),
\ee
then, with probability at least $1-\delta$ over $M\sim\kacdist_{n,T,x}$,
\be \label{EqJLKacConclusion}
(1-\varepsilon)\|x_i-x_j\|_2^2
\leq
\|A_M(x_i-x_j)\|_2^2
\leq
(1+\varepsilon)\|x_i-x_j\|_2^2
\ee
for every pair $1\leq i<j\leq N$.

In particular, for fixed $\varepsilon$ and $N/\delta\leq n^A$, with $A<\infty$ fixed, the usual Johnson--Lindenstrauss target dimension $\ell=O(\varepsilon^{-2}\log(N/\delta))=O_A(\varepsilon^{-2}\log n)$ \vnote{Isn't the JL target dimension $\epsilon^{-2}\log N$? EDIT: Ah, I see what's going on, $N$ is set to polynomial in $n$. But see below.} is achieved by taking $T\geq C_A n\log^3 n$. \vnote{I am worried that this is not the best illustration of our bound. In particular, when $N$ is polynomial in $n$, the JL embedding dimension $\ell$ is already $O(\log n)$ and so applying the matrix already only costs $O(n\log n)$. Shouldn't we be thinking about very large point sets of size, say $2^{n^{\epsilon'}}$ for some constant $0 < \epsilon' < 1$.}
\end{theorem}

When $x=I_n$, the map $v\mapsto A_Mv$ can be applied by successively applying the $T$ elementary rotations and then projecting to the first $\ell$ coordinates. Thus the arithmetic cost per vector is $O(T+\ell)$; in the polynomial-size regime described above this is $\tilde{O}(n)$, compared with the $O(n\ell)$ cost of applying a dense $\ell\times n$ projection matrix.
The proof is given in Section \ref{SecProofFastJL}.

\section{Open Questions}

\subsection{Immediate Open Questions} 

There are several natural open questions that come out of our work:

\begin{itemize}
    \item Can the $k$-column mixing result be extended to TV distance and not just Wasserstein? This appears plausible, and we leave it for future work.
    \item Can we tighten the polynomial indistinguishability result? First, recall that our Theorem~\ref{MainThm} stops being interesting for $k \approx \sqrt{n}$: can we extend this all the way to $k \approx n$? Secondly, to be indistinguishable against degree-$k$ polynomials, we need the walk to run for roughly $nk^2$ steps: is that necessary?
    \item Finally, a very natural question is whether low-degree polynomials of the {\em spectrum} of a Kac matrix can be a good distinguisher. It seems plausible that the answer is no: the basic intuition here is just that the spectral measure is a ``nice" function of only n ``pretty independent" numbers, so should mix in $n \mathsf{poly}(\log n)$ steps. However, making this argument formally appears more difficult, since the spectrum is not quite a Markov chain (in contrast to the columns that we analyze here, which are).
\end{itemize}

\subsection{Trapdoored Functions}


While this work is about trapdoor matrices, one could investigate trapdoor functions more generally. What we really want is a random walk on functions, so that (i) it is easy to evaluate (trapdoored) functions based on short walks, and furthermore (ii) the functions can ``hide" in a bigger class of functions that cannot be evaluated quickly. That is, we have a ``big" space of circuits that are mostly ``slow," then we hide a ``small" space of circuits that are all ``fast" inside it. For matrix multiplication, something magic happens in part (ii) --- compositions of matrices are matrices, so the hiding is essentially automatic. Are there other classes of functions for which such hiding is automatic? Do any of them have worse complexity than $n^{2}$?

In one extreme, when the class of functions is all functions say from $\{0,1\}^n \to \{0,1\}^n$, the cryptographic notion of pseudorandom functions~\cite{DBLP:journals/jacm/GoldreichGM86} provides a solution. But how about, say, when the class of functions is all functions that can be computed by circuits of size $n^{100}$? Can circuits of size $n^{5}$ hide amongst them?

\vnote{TODO for Aaron}

\vnote{Mention Aaron's concurrent example. Also mention RRJW.}

\anote{I'm very interested in this, and have a paper draft ``on hold" in this direction. However, I think it is somewhat subsumed in Vinod's existing heuristic paper, so I'm inclined to delete it.]}

\anote{The following is all very vague, but interesting at least to me. This paper sort of suggests that (i) maybe it is OK to just let some functions of the chain mix, and (ii) that this is a meaningful speedup. That seems to be against most of the conventional wisdom I've seen, and it would be interesting to see if there were some sort of nice theory about when (and if) this phenomenon is real. My personal case studies here are other MCMC algorithms that sort of look like big matrices - e.g. probabilistic dimension reduction. Actually I think this phenomenon is basically already in evidence if you look at simulations of e.g. the LPM in Rastelli et al. ]}

\bibliographystyle{alpha}
\bibliography{ref}

\appendix


\section{Some Geometric Calculations}

We collect some facts about the geometry of $\So{n}$ and $\mathrm{St}(n,k)$. We suspect everything in this section is well-known, but include the details for completeness.

\subsection{Exact row identities on the Stiefel manifold} \label{SecPainfulAlg}

The local contraction proof in Lemma \ref{LemmaLocalCont} uses the following secant identity.

\begin{lemma}[Secant row-sum identity]
\label{lem:row_sum_secant_identity}
Let $X,Y\in \mathrm{St}(n,k)$ and set $K=Y-X$. Define
\[
q_{ij}=r_i(X)\cdot r_j(K)-r_j(X)\cdot r_i(K),\qquad 1\leq i<j\leq n.
\]
Then
\be\label{EqSecantRowSumIdentity}
\sum_{i<j}q_{ij}^{2}
=\|K\|_{F}^{2}-\mathrm{Tr}\bigl((K^{\top}X)^{2}\bigr).
\ee
Moreover,
\be\label{EqSecantRowSumLower}
\sum_{i<j}q_{ij}^{2}
\geq \|K\|_{F}^{2}-\frac14\|K\|_{F}^{4}.
\ee
In particular, if $\|K\|_{F}\leq 1$, then
\[
\sum_{i<j}q_{ij}^{2}\geq \frac34\|K\|_{F}^{2}.
\]
\end{lemma}

\begin{proof}
Set $A=KX^\top\in\mathbb{R}^{n\times n}$. Then $A_{ij}=r_i(K)\cdot r_j(X)$, and hence
\[
\sum_{i<j}q_{ij}^{2}
=\sum_{i<j}(A_{ij}-A_{ji})^2
=\frac12\|A-A^\top\|_F^2.
\]
Using $X^\top X=I_k$ and cyclicity of trace,
\[
\frac12\|A-A^\top\|_F^2
=\|A\|_F^2-\mathrm{Tr}(A^2)
=\|K\|_F^2-\mathrm{Tr}\bigl((K^\top X)^2\bigr),
\]
which proves \eqref{EqSecantRowSumIdentity}.

Since $Y=X+K$ and $X^\top X=Y^\top Y=I_k$,
\[
X^\top K+K^\top X=-K^\top K.
\]
Thus, writing $B=K^\top X=A_0+S_0$ with $A_0$ skew-symmetric and $S_0$ symmetric, we have $S_0=-K^\top K/2$. Therefore
\[
\mathrm{Tr}(B^2)=\mathrm{Tr}(A_0^2)+\mathrm{Tr}(S_0^2)\leq \mathrm{Tr}(S_0^2)
=\frac14\|K^\top K\|_F^2\leq \frac14\|K\|_F^4.
\]
Plugging this into \eqref{EqSecantRowSumIdentity} proves \eqref{EqSecantRowSumLower}.
\end{proof}

\subsection{Tail Bounds Supnorm of Haar-distributed Points on Sphere}

The following estimate is standard but somewhat tedious to write down. Since we only need rather rough estimates, we give a sketch:

\begin{lemma}
\label{lem:haar_sphere_supnorm}
Let $Y$ be uniform on $\mathbb{S}^{n-1}$. There is a universal constant $C>0$ such that
for all $\delta\in(0,1)$,
\[
\mathbb{P}\!\left[\|Y\|_\infty^2 \ge C\,\frac{\log(n/\delta)}{n}\right]\le \delta.
\]
\end{lemma}

\begin{proof}[Proof sketch]
Write $Y=Z/\|Z\|_2$ with $Z_i\stackrel{iid}{\sim}N(0,1)$. By e.g. Theorem 4.1 of \cite{Devroye1986}, $Y$ is distributed according to the Haar measure on the sphere. We can use (i) a union bound and standard Gaussian tail estimates to get an upper bound on
$\max_i |Z_i|$, and (ii) similar $\chi^{2}$ tail estimates to get a lower-tail bound for $\|Z\|_2$.
\end{proof}

\begin{lemma}[Haar projection tail]\label{lem:haar_projection_tail}
Let $Y$ be uniform on $\mathbb{S}^{n-1}$, and let $V\subseteq\mathbb{R}^{n}$ be a fixed subspace with $\dim(V)=k$. There is a universal constant $C>0$ such that, for all $s\geq 1$,
\[
\mathbb{P}\left[\|P_VY\|^{2}> C\,\frac{k+s}{n}\right]\leq e^{-s}.
\]
\end{lemma}

\begin{proof}[Proof sketch]
Write $Y=Z/\|Z\|$ with $Z\sim N(0,I_n)$, and rotate coordinates so that $V=\mathrm{span}(e_1,\ldots,e_k)$. Then $\|P_VZ\|^2$ is $\chi^2_k$ and $\|Z\|^2$ is $\chi^2_n$. Standard upper-tail bounds for $\chi^2_k$, together with the lower-tail bound $\mathbb{P}[\|Z\|^2<n/2]\leq e^{-cn}$, imply the displayed estimate after increasing $C$. If $s$ is larger than a constant multiple of $n$, the threshold is already at least $1$ after increasing $C$, so the estimate is trivial.
\end{proof}

\subsection{Metric comparisons}\label{SecLiftPts}

\begin{lemma}[Metric comparison on the Stiefel manifold]
\label{lem:stiefel_metric_comparison}
Let $D$ be the Riemannian distance on $\mathrm{St}(n,k)$ induced by the Frobenius
inner product from the embedding $\mathrm{St}(n,k)\subset \mathbb{R}^{n\times k}$.
Then, for all $X,Y\in \mathrm{St}(n,k)$,
\begin{equation}
\label{eq:stiefel_chordal_leq_riemannian}
\|X-Y\|\leq D(X,Y).
\end{equation}
Moreover, there are universal constants $\rho_D>0$ and $C_D<\infty$, independent
of $n$ and $k$, such that whenever $\|X-Y\|\leq \rho_D$,
\begin{equation}
\label{eq:stiefel_riemannian_local_upper}
D(X,Y)\leq \|X-Y\|+C_D\|X-Y\|^2.
\end{equation}
\end{lemma}

\begin{proof}
We first prove \eqref{eq:stiefel_chordal_leq_riemannian}. If
$\gamma:[0,1]\to \mathrm{St}(n,k)$ is any piecewise $C^1$ curve with
$\gamma(0)=X$ and $\gamma(1)=Y$, then
\begin{equation}
\label{eq:curve_length_bounds_chord}
\|X-Y\|
=
\left\|\int_0^1 \dot\gamma(t)\,dt\right\|
\leq
\int_0^1 \|\dot\gamma(t)\|\,dt.
\end{equation}
Taking the infimum over all such curves gives
\eqref{eq:stiefel_chordal_leq_riemannian}.

We now prove the upper bound \eqref{eq:stiefel_riemannian_local_upper} by constructing a specific path from
$X$ to $Y$ and bounding  its length.

Set
\begin{equation}
\label{eq:metric_comparison_K_delta}
K=Y-X,\qquad \delta=\|K\|.
\end{equation}
We will assume $\delta\leq 1$. Since $X^\top X=Y^\top Y=I_k$ and $Y=X+K$, we have
\begin{equation}
\label{eq:metric_comparison_secant_relation}
X^\top K+K^\top X+K^\top K=0.
\end{equation}
Let
\begin{equation}
\label{eq:metric_comparison_chord}
Z_t=X+tK,\qquad 0\leq t\leq 1,
\end{equation}
and write its Gram matrix as
\begin{equation}
\label{eq:metric_comparison_gram}
G_t=Z_t^\top Z_t
=
I_k-t(1-t)K^\top K,
\end{equation}
where the last equality follows from
\eqref{eq:metric_comparison_secant_relation}. Since
$t(1-t)\leq 1/4$ and $\|K^\top K\|_{\mathrm{op}}\leq \delta^2$, every eigenvalue of
$G_t$ lies in $[1-\delta^2/4,1]$. In particular, if $\delta\leq 1$, then $G_t$ is
positive definite for every $t\in[0,1]$.

We claim that this allows us to define a path in $\mathrm{St}(n,k)$ via the following polar normalization:
\begin{equation}
\label{eq:metric_comparison_polar_path}
\gamma(t)=Z_tG_t^{-1/2},
\end{equation}
where $G_t^{-1/2}$ denotes the principal inverse square root. The map
$G\mapsto G^{-1/2}$ is smooth on the positive definite cone, and the polar-factor map
$Z\mapsto Z(Z^\top Z)^{-1/2}$ is smooth on the open set of full-column-rank matrices;
see \cite[Chapters~6 and~8]{high:FM}. Thus $\gamma$ is a smooth path. Moreover,
\[
\gamma(t)^\top \gamma(t)
=
G_t^{-1/2}G_tG_t^{-1/2}
=
I_k,
\]
so $\gamma(t)\in \mathrm{St}(n,k)$ for every $t$. Since $G_0=G_1=I_k$, we also have
$\gamma(0)=X$ and $\gamma(1)=Y$.

It remains to bound the length of this path. Because $G_t$ is a polynomial in the fixed
symmetric matrix $K^\top K$, the matrices $G_t$ and $\dot G_t$ commute. Differentiating
\eqref{eq:metric_comparison_polar_path} along this commuting family gives
\begin{equation}
\label{eq:metric_comparison_gamma_derivative}
\dot\gamma(t)
=
K G_t^{-1/2}
+
\frac{1-2t}{2}\,Z_t K^\top K G_t^{-3/2}.
\end{equation}
Equivalently, this follows by diagonalizing $K^\top K$ and applying the scalar derivative
of $u\mapsto u^{-1/2}$ to each eigenvalue; this is a special case of the matrix-function
calculus discussed in \cite[Chapter~6]{high:FM}.

From the eigenvalue bound following \eqref{eq:metric_comparison_gram}, for
$\delta\leq 1$ we have
\begin{equation}
\label{eq:metric_comparison_inverse_bounds}
\|G_t^{-1/2}\|_{\mathrm{op}}
\leq
(1-\delta^2/4)^{-1/2}
\leq
1+\delta^2,
\qquad
\|G_t^{-3/2}\|_{\mathrm{op}}
\leq
(1-\delta^2/4)^{-3/2}
\leq
2.
\end{equation}
Also,
\begin{equation}
\label{eq:metric_comparison_Z_K_bounds}
\|Z_t\|_{\mathrm{op}}^2=\|G_t\|_{\mathrm{op}}\leq 1,
\qquad
\|K^\top K\|_F\leq \|K\|^2=\delta^2.
\end{equation}
Combining \eqref{eq:metric_comparison_gamma_derivative},
\eqref{eq:metric_comparison_inverse_bounds}, and
\eqref{eq:metric_comparison_Z_K_bounds}, and using
$\|AB\|_F\leq \|A\|_F\|B\|_{\mathrm{op}}$, gives
\begin{equation}
\label{eq:metric_comparison_speed_bound}
\|\dot\gamma(t)\|
\leq
\delta(1+\delta^2)+|1-2t|\delta^2.
\end{equation}
Therefore
\begin{equation}
\label{eq:metric_comparison_length_bound}
D(X,Y)
\leq
\int_0^1 \|\dot\gamma(t)\|\,dt
\leq
\delta(1+\delta^2)+\frac12\delta^2.
\end{equation}
Since $\delta\leq 1$, this implies
\[
D(X,Y)
\leq
\delta+2\delta^2
=
\|X-Y\|+2\|X-Y\|^2.
\]
Thus \eqref{eq:stiefel_riemannian_local_upper} holds, for example, with
$\rho_D=1$ and $C_D=2$.
\end{proof}

\subsection{Deterministic non-expansion in Frobenius Norm}\label{AppendixSillyCalc1}

\begin{lemma}[Deterministic Frobenius non-expansion]
\label{lem:one_step_stability}
Let $X,Y\in \mathrm{St}(n,k)$. Run one step of the coupled update
\[
X'=R_{ij}(\theta)X,\qquad Y'=R_{ij}(\theta+\alpha)Y,
\]
with $\alpha=\widehat\alpha_{ij}(X,Y)$ chosen as in \eqref{EqAlphaOptimal}. Then deterministically
\[
\|X'-Y'\|\leq \|X-Y\|.
\]
\end{lemma}

\begin{proof}
Applying the common inverse rotation $R_{ij}(-\theta)$ to both updated points gives
\[
\|X'-Y'\|=\|X-R_{ij}(\alpha)Y\|.
\]
By \eqref{EqTwoRowDistanceFormula} and \eqref{EqAlphaOptimal}, the choice $\alpha=\widehat\alpha_{ij}(X,Y)$ minimizes $\|X-R_{ij}(\alpha)Y\|$ over all angle shifts in the selected coordinate plane. Comparing to the admissible choice $\alpha=0$ gives
\[
\|X'-Y'\|\leq \|X-Y\|.
\]
\end{proof}

\section{Miscellaneous Estimates}

We present useful estimates that we believe are essentially well-known, including proofs for completeness. The notation in this section is independent of the rest of the paper.

\subsection{Telescoping Sum Inequality}

\begin{lemma}\label{LemmaTelescopingSum}
Let $X=(X_{1},\dots,X_{n})$ and $Y=(Y_{1},\dots,Y_{n})$ be vectors in $\mathbb{R}^{n}$. Then
\[
\prod_{i=1}^{n} X_{i}-\prod_{i=1}^{n} Y_{i}
=
\sum_{k=1}^{n}
(X_{k}-Y_{k})
\left(\prod_{i=1}^{k-1} X_{i}\right)
\left(\prod_{i=k+1}^{n} Y_{i}\right).
\]
\end{lemma}

\begin{proof}
Write
\be 
\prod_{i=1}^{n} X_{i}-\prod_{i=1}^{n} Y_{i}
=
\sum_{k=1}^{n}
\left(
\prod_{i=1}^{k} X_{i}\prod_{i=k+1}^{n} Y_{i}
-
\prod_{i=1}^{k-1} X_{i}\prod_{i=k}^{n} Y_{i}
\right),
\ee
in which all intermediate terms cancel. Factoring out
\be
\prod_{i=1}^{k-1} X_{i}\prod_{i=k+1}^{n} Y_{i}
\ee
from each summand finishes the proof.
\end{proof}

\subsection{Contraction with occasional bad non-expansive steps}

\begin{lemma}
\label{lem:B1_multiplicative}
Let $(\mathcal{L}_t)_{t=0}^{T}$ be a nonnegative process adapted to a filtration $(\mathscr F_t)_{t=0}^{T}$. Let $\mathcal G\subseteq \mathscr F_0$ be an initial sigma-field, and let $B$ be a nonnegative $\mathcal G$-measurable random variable such that $\mathcal{L}_0\leq B$. Fix $\kappa\in(0,1]$, and let $G_0,\ldots,G_{T-1}$ be events with $G_t\in\mathscr F_t$.

Assume that, for every $t=0,\ldots,T-1$:
\begin{enumerate}
\item (\textbf{Deterministic non-expansion}) $\mathcal{L}_{t+1}\leq \mathcal{L}_t$ almost surely.
\item (\textbf{Good-step contraction}) On $G_t$,
\[
\mathbb{E}[\mathcal{L}_{t+1}\mid \mathscr F_t]\leq (1-\kappa)\mathcal{L}_t.
\]
\item (\textbf{Conditional bad-event control}) For deterministic numbers $\delta_t\geq0$,
\[
\mathbb{P}(G_t^c\mid \mathcal G)\leq \delta_t.
\]
\end{enumerate}
Then
\be\label{EqConditionalVariableRateConclusion}
\mathbb{E}[\mathcal{L}_T\mid\mathcal G]
\leq \left((1-\kappa)^T+\sum_{t=0}^{T-1}\delta_t\right)B.
\ee
\end{lemma}

\begin{proof}
Deterministic non-expansion implies $\mathcal{L}_t\leq B$ for every $t$. The first two assumptions imply
\[
\mathbb{E}[\mathcal{L}_{t+1}\mid\mathscr F_t]
\leq (1-\kappa)\mathcal{L}_t+B\mathbf 1_{G_t^c}.
\]
Taking conditional expectations with respect to $\mathcal G$ and using the conditional bad-event bound gives
\[
\mathbb{E}[\mathcal{L}_{t+1}\mid\mathcal G]
\leq (1-\kappa)\mathbb{E}[\mathcal{L}_t\mid\mathcal G]+\delta_t B.
\]
Since $\mathcal{L}_0\leq B$, iterating this recursion proves \eqref{EqConditionalVariableRateConclusion}; in the iteration we simply bound the geometric weights multiplying the error terms by $1$.
\end{proof}

\section{Burn-in Estimates} \label{SecBurninEstimatesAll}

We give a sequence of estimates, with the goal of showing that certain functions of Kac's walk are reasonably well-mixed after a short burn-in period.

\subsection{Burn-in for 1-column Walk}\label{AppLemmaNormBounds}

\begin{lemma}
\label{lem:onecol_supnorm_burnin}
Let $\{X_t\}$ be the $1$-column Kac walk on $\mathbb{S}^{n-1}$ started from any $X_0$.
For every $B<\infty$ there is a constant $C_B<\infty$ such that, for every $\eta\in(0,1/2)$ and every
\[
t\geq C_B n\log(n/\eta),
\]
one can couple $X_t$ to a Haar-distributed point $Y\in\mathbb{S}^{n-1}$ so that
\[
\mathbb{P}[\|X_t-Y\|>n^{-B}]\leq \eta.
\]
\end{lemma}

\begin{proof}
By Lemma 3.3 of \cite{pillaiKac17} and Markov's inequality, it is possible to couple $X_t,Y$ so that
\[
\mathbb{P}\left[\frac{\|X_t-Y\|_\infty}{\min_i |Y[i]|}>2e^{-t/(3n)}\right]\leq e^{-t/3}.
\]
By Lemma 3.5 of the same paper,
\[
\mathbb{P}\left[\min_i |Y[i]|< e^{-t/(6n)}\right]\leq 2ne^{-t/(18n)}.
\]
Thus
\[
\mathbb{P}[\|X_t-Y\|_\infty>2e^{-t/(6n)}]
\leq (2n+1)e^{-t/(18n)}.
\]
Choosing $C_B$ large enough makes the last probability at most $\eta$ and also makes $\sqrt n\,2e^{-t/(6n)}\leq n^{-B}$.
\end{proof}

\subsection{Burn-in for \texorpdfstring{$k$-}{k-}Column Walk}

\begin{lemma}
\label{lem:kcol_rowmax_burnin}
Let $\{X_t\}$ be the $k$-column walk on $\mathrm{St}(n,k)$ started from any $X_0$, with $1\leq k<n$. There are universal constants $C,c_0<\infty$ such that, for all $\delta\in(0,1/2)$ and all
\[
t\geq c_0 n\log(n/\delta),
\]
we have
\[
\mathbb{P}\!\left[\|X_t\|_{row,\infty}^2 \ge C\,\frac{k+\log(n/\delta)}{n}\right]\le \delta.
\]
\end{lemma}

\begin{proof}
Write $X_t=Q_tX_0$, where $Q_t$ is the product of the elementary rotations used by the walk, and let $V_0=\mathrm{span}(X_0)$. For a fixed row $a$,
\[
\|r_a(X_t)\|^2=\|X_0^\top Q_t^\top e_a\|^2=\|P_{V_0}Q_t^\top e_a\|^2.
\]
The random vector $Q_t^\top e_a$ has the same distribution as the $1$-column Kac walk started from $e_a$: transposition reverses the order of the rotations and replaces each angle by its negative, which does not change the joint law.

Apply Lemma~\ref{lem:onecol_supnorm_burnin} with $B=10$ and $\eta=\delta/(2n)$. If $U$ is Haar on $\mathbb{S}^{n-1}$, then for $t\geq c_0n\log(n/\delta)$ we may couple $Q_t^\top e_a$ and $U$ so that the coupling error has probability at most $\delta/(2n)$. On the complement of this error event,
\[
\|P_{V_0}Q_t^\top e_a\|^2
\leq 2\|P_{V_0}U\|^2+2n^{-20}.
\]
By Lemma~\ref{lem:haar_projection_tail}, with $s=\log(2n/\delta)$,
\[
\mathbb{P}\left[\|P_{V_0}U\|^2>C\frac{k+\log(2n/\delta)}{n}\right]\leq \frac{\delta}{2n}.
\]
Increasing $C$ absorbs the $2n^{-20}$ term. Thus the desired row bound fails for this fixed $a$ with probability at most $\delta/n$. A union bound over $a\in[n]$ completes the proof.
\end{proof}

\subsection{Contraction Over Epochs}

We combine the conditional contraction estimate in Lemma \ref{LemmaLocalCont}, the deterministic non-expansion in Lemma \ref{lem:one_step_stability}, and the burn-in estimate in Lemma \ref{lem:kcol_rowmax_burnin} to obtain a local contraction estimate for a power of the transition kernel. The parameter choices are collected in the following definition.

\begin{defn}
\label{def:C4_epoch}
Set $\delta_{\mathrm{step}} = n^{-30}$. Define $M$ by
\[
M = C\,\frac{k+\log\big(2n/\delta_{\mathrm{step}}\big)}{n},
\]
where $C>0$ is chosen large enough for Lemma \ref{lem:kcol_rowmax_burnin}, and let $c_{\mathrm{loc}}>0$ be the universal constant from Lemma \ref{LemmaLocalCont}. Define
\[
\kappa = \frac{c_{\mathrm{loc}}}{Mn^2}.
\]
Finally, set the epoch lengths
\[
T_{\mathrm{burn}} = \left\lceil c_0 n\log\big(2n/\delta_{\mathrm{step}}\big) \right\rceil,
\qquad
T_{\mathrm{contr}} = \left\lceil \frac{20}{\kappa}\,\log n\right\rceil,
\qquad
T_{\mathrm{epoch}} = T_{\mathrm{burn}}+T_{\mathrm{contr}}.
\]
\end{defn}

\begin{lemma}\label{lem:epoch_contraction_v2_refactored}
Fix notation as in Definition~\ref{def:C4_epoch}, let $P$ be the one-step transition kernel of the $k$-column walk, and let $W_D$ denote Wasserstein distance with cost $D$. There is a universal constant $\rho_{\mathrm{loc}}>0$, independent of $n$ and $k$, such that $\rho_{\mathrm{loc}}\leq \min(1,\rho_D)$ and the following holds. If $X_0,Y_0\in \mathrm{St}(n,k)$ satisfy $D(X_0,Y_0)\leq \rho_{\mathrm{loc}}$, then
\be\label{eq:epoch_boundary_contraction}
W_D\left(P^{T_{\mathrm{epoch}}}(X_0,\cdot),P^{T_{\mathrm{epoch}}}(Y_0,\cdot)\right)
\le (1+C_D\rho_{\mathrm{loc}})\Big((1-\kappa)^{T_{\mathrm{contr}}}+\delta_{\mathrm{epoch}}\Big)D(X_0,Y_0),
\ee
where $\delta_{\mathrm{epoch}}=T_{\mathrm{contr}}\delta_{\mathrm{step}}$. In particular, for all $n > N_{0}$ sufficiently large,
\be \label{eq:epoch_boundary_contraction_final_conclusion}
W_D\left(P^{T_{\mathrm{epoch}}}(X_0,\cdot),P^{T_{\mathrm{epoch}}}(Y_0,\cdot)\right)
\le n^{-1}D(X_0,Y_0).
\ee
\end{lemma}

\begin{proof}
We construct a coupling realizing the stated bound. Starting from $X_0,Y_0$, use the coupling defined in Equations \eqref{EqCoupling}--\eqref{EqAlphaOptimal} at every step. Let $\mathscr F_t$ be the natural filtration and set
\[
\mathcal{L}_t=\|X_t-Y_t\|.
\]
By \eqref{IneqLocalEquivFrobRiem} and the assumption on $D(X_0,Y_0)$,
\[
\mathcal{L}_0\leq D(X_0,Y_0)\leq \rho_{\mathrm{loc}}\leq 1.
\]
Lemma \ref{lem:one_step_stability} gives the deterministic non-expansion
\be\label{EqLocalEpochNonExpansion}
\mathcal{L}_{t+1}\leq \mathcal{L}_t
\qquad \text{for every }t.
\ee
In particular, $\mathcal{L}_t\leq \rho_{\mathrm{loc}}$ for the whole epoch, so the local contraction lemma remains applicable whenever the row-norm parameter is at most $M$.

For $t\geq T_{\mathrm{burn}}$, define
\[
G_t(M)=\left\{\max\big(\|X_t\|_{row,\infty}^2,\|Y_t\|_{row,\infty}^2\big)\leq M\right\}.
\]
By Lemma \ref{lem:kcol_rowmax_burnin}, the Markov property, and a union bound over the two endpoints of the chains,
\be\label{EqEndpointRowNormGood}
\mathbb{P}(G_t(M)^c)\leq \delta_{\mathrm{step}}
\qquad \text{for every }t\geq T_{\mathrm{burn}}.
\ee
On $G_t(M)$, Lemma \ref{LemmaLocalCont} gives
\be\label{EqLocalEpochGoodStep}
\mathbb{E}[\mathcal{L}_{t+1}\mid \mathscr F_t]\leq (1-\kappa)\mathcal{L}_t.
\ee
The burn-in portion is deterministically non-expanding, so $\mathcal{L}_{T_{\mathrm{burn}}}\leq D(X_0,Y_0)$. Apply Lemma \ref{lem:B1_multiplicative} to the shifted process $\mathcal{L}_{T_{\mathrm{burn}}+u}$, $u=0,\ldots,T_{\mathrm{contr}}$, with $B=D(X_0,Y_0)$, the trivial initial sigma-field, and events $G_{T_{\mathrm{burn}}+u}(M)$. Using \eqref{EqLocalEpochNonExpansion}--\eqref{EqLocalEpochGoodStep}, we obtain
\[
\mathbb{E}[\mathcal{L}_{T_{\mathrm{epoch}}}]
\leq \Big((1-\kappa)^{T_{\mathrm{contr}}}+\delta_{\mathrm{epoch}}\Big)D(X_0,Y_0).
\]
Because $\mathcal{L}_{T_{\mathrm{epoch}}}\leq \rho_{\mathrm{loc}}\leq \rho_D$ deterministically, the local metric comparison \eqref{IneqLocalUpperFrobRiem} gives
\[
D(X_{T_{\mathrm{epoch}}},Y_{T_{\mathrm{epoch}}})
\leq (1+C_D\rho_{\mathrm{loc}})\mathcal{L}_{T_{\mathrm{epoch}}}.
\]
Taking expectations proves \eqref{eq:epoch_boundary_contraction}.

Finally, $T_{\mathrm{contr}}=\lceil20\kappa^{-1}\log n\rceil$ gives $(1-\kappa)^{T_{\mathrm{contr}}}\leq n^{-20}$. Since $M=O((k+\log n)/n)$ and $1\leq k<n$, we have $T_{\mathrm{contr}}=O(n(k+\log n)\log n)\leq O(n^2\log n)$, so $\delta_{\mathrm{epoch}}=o(n^{-10})$. The prefactor $1+C_D\rho_{\mathrm{loc}}$ is a fixed constant, and \eqref{eq:epoch_boundary_contraction_final_conclusion} follows for all sufficiently large $n$. The same estimates also show $T_{\mathrm{epoch}}\leq Cn(k+\log n)\log n$ after increasing $C$.
\end{proof}

\section{Representation-Theoretic Variance Bounds for Haar Polynomials}
\label{SecRepTheoryVariance}

We prove the main representation-theory result for any compact group and then later
transfer it to $\So{n}$. We also provide context in the broader representation-theory literature. 

We recall the following definition (see \citep[Definition 2.2]{Sepan}):
\begin{defn}[Isotypic component]
Let $V$ be a unitary representation of $G$.
For an irreducible representation $E_\pi$, define $V[\pi]$ to be the largest
subspace of $V$ that is a direct sum of irreducible subrepresentations
equivalent to $E_\pi$. The sub-module $V[\pi]$
is called the $\pi$-isotypic component of $V$.
\end{defn}

We also recall the following result (see \citep[Theorem 2.24]{Sepan}): \anote{We don't define some words, like ``intertwining." I think that's probably OK, but maybe we should acknowledge it.]}
\begin{theorem}[Canonical Decomposition] \label{thm:cd}
Let $V$ be a finite-dimensional representation of a compact Lie group $G$.

\begin{enumerate}
\item There is a $G$-intertwining isomorphism
\[
\iota_\pi :
\mathrm{Hom}_G(E_\pi, V)\otimes E_\pi
\;\xrightarrow{\ \sim\ }\;
V[\pi]
\]
induced by mapping $T \otimes v \mapsto T(v)$ for $T \in \mathrm{Hom}_G(E_\pi, V)$ and $v \in V$.
In particular, the multiplicity of $\pi$ is
\[
m_\pi = \dim \mathrm{Hom}_G(E_\pi, V).
\]

\item There is a $G$-intertwining isomorphism
\[
\bigoplus_{[\pi]\in \widehat{G}} \mathrm{Hom}_G(E_\pi, V)\otimes E_\pi
\;\xrightarrow{\ \sim\ }\;
V
=
\bigoplus_{[\pi]\in \widehat{G}} V[\pi].
\]
\end{enumerate}
\end{theorem}

\begin{theorem}[Compact-group variance inequality]
\label{thm:compact-group-variance}
Let $G$ be a compact Lie group, let $\mu$ be Haar probability measure on $G$, and let
$U$ be a finite-dimensional real inner-product space. Let
\[
\rho:G\to \End(U_\C)
\]
be the complexification of an orthogonal representation of $G$ on $U$.
For $T\in \End(U_\C)$ define
\[
\Phi(T)(g):=\langle T,\rho(g)\rangle_{\HS},
\qquad
\Psi(T):=\Phi(T)-\int_G \Phi(T)(h)\,d\mu(h).
\]
Let
\[
\mathcal C:=\{T\in \End(U_\C): \Phi(T)\text{ is constant on }G\}.
\]
Then
\[
\|\Psi(T)\|_{L^2(G,\mu)}^2
\ge
\frac{1}{\dim U}\,\| \mathrm{Proj}_{\mathcal C^\perp}T\|_{\HS}^2.
\]
\end{theorem}

\begin{proof}
By the canonical decomposition in Theorem \ref{thm:cd},
the complexified representation admits a $G$-equivariant decomposition
\[
U_\C \;\cong\; \bigoplus_{\lambda\in\widehat G}
V_\lambda \otimes M_\lambda,
\qquad
M_\lambda := \mathrm{Hom}_G(V_\lambda, U_\C),
\]
where $V_\lambda$ runs over irreducible representations and $M_\lambda$
is the multiplicity space.

We now describe the action of $\rho(g)$ under this identification. By Theorem~\ref{thm:cd}, the isomorphism is given by
\[
\iota_\lambda(v \otimes T) = T(v), \qquad v\in V_\lambda,\ T\in M_\lambda.
\]
Since $T$ is $G$-equivariant, for every $g\in G$,
\[
\rho(g)\,T(v) = T(\rho_\lambda(g)v).
\]
Therefore
\[
\rho(g)\,\iota_\lambda(v \otimes T)
=
\iota_\lambda\bigl(\rho_\lambda(g)v \otimes T\bigr),
\]
which shows that, under the above identification,
\[
\rho(g)\big|_{V_\lambda\otimes M_\lambda}
=
\rho_\lambda(g)\otimes I_{M_\lambda}.
\]
Hence
\[
\rho(g)\;\cong\;\bigoplus_{\lambda\in\widehat G}\rho_\lambda(g)\otimes I_{M_\lambda}.
\]
Write
\[
d_\lambda:=\dim V_\lambda,\qquad m_\lambda:=\dim M_\lambda.
\]
With respect to this decomposition, every
\[
T\in \End(U_\C)
\]
has block form
\[
T=(T_{\lambda\kappa}),
\qquad
T_{\lambda\kappa}\in \mathrm{Hom}(V_\lambda\otimes M_\lambda,\;V_\kappa\otimes M_\kappa).
\]
Since $\rho(g)$ is block diagonal, only the diagonal blocks contribute to
$\langle T,\rho(g)\rangle_{\HS}$:
\[
\Phi(T)(g)
=
\sum_{\lambda\in\widehat G}
\bigl\langle T_{\lambda\lambda},\,\rho_\lambda(g)\otimes I_{M_\lambda}\bigr\rangle_{\HS}.
\]
For each $\lambda$, define a linear map
\[
L_\lambda:\End(V_\lambda\otimes M_\lambda)\to \End(V_\lambda)
\]
as follows. If $\{f_a\}$ is an orthonormal basis of $M_\lambda$, set
\begin{equation} \label{eqn:LL}
\langle L_\lambda(A)u,v\rangle
:=
\sum_a \langle A(u\otimes f_a),\,v\otimes f_a\rangle,
\qquad u,v\in V_\lambda.
\end{equation}
Thus $L_\lambda(A) = \mathrm{Trace}_{M_\lambda}(A)$ is 
the partial trace over \(M_\lambda\).
It is independent of the choice of orthonormal basis of $M_\lambda$.
Set
\[
B_\lambda:=L_\lambda(T_{\lambda\lambda})\in \End(V_\lambda).
\]

We now compute the contribution of the $\lambda$-block to $\Phi(T)$.
Fix orthonormal bases $\{e_i\}$ of $V_\lambda$ and $\{f_a\}$ of $M_\lambda$. Then
\begin{align*}
\bigl\langle T_{\lambda\lambda},\,\rho_\lambda(g)\otimes I_{M_\lambda}\bigr\rangle_{\HS}
&=
\sum_{i,a}
\bigl\langle T_{\lambda\lambda}(e_i\otimes f_a),\,(\rho_\lambda(g)e_i)\otimes f_a\bigr\rangle \\
&=
\sum_{i,j,a}
\bigl\langle T_{\lambda\lambda}(e_i\otimes f_a),\,e_j\otimes f_a\bigr\rangle
\bigl\langle \rho_\lambda(g)e_i,e_j\bigr\rangle \\
&=
\sum_{i,j}
\langle B_\lambda e_i,e_j\rangle
\bigl\langle \rho_\lambda(g)e_i,e_j\bigr\rangle \\
&=
\langle B_\lambda,\rho_\lambda(g)\rangle_{\HS}.
\end{align*}
Hence
\[
\Phi(T)(g)=\sum_{\lambda\in\widehat G}\langle B_\lambda,\rho_\lambda(g)\rangle_{\HS}.
\]
Next we identify the centered part. We have
\[
\int_G \Phi(T)(h)\,d\mu(h)
=
\sum_{\lambda\in\widehat G}
\left\langle B_\lambda,\int_G \rho_\lambda(h)\,d\mu(h)\right\rangle_{\HS}
=
\sum_{\lambda\in\widehat G}\langle B_\lambda,A_\lambda\rangle_{\HS}
\]
with
\[
A_\lambda:=\int_G \rho_\lambda(h)\,d\mu(h)\in \End(V_\lambda).
\]
Since $A_\lambda$ commutes with $\rho_\lambda(g)$ for every $g\in G$, Schur's lemma
gives $A_\lambda=c_\lambda I_{V_\lambda}$ for some scalar $c_\lambda$.

If $\lambda$ is nontrivial, then
\[
A_\lambda:=\int_G \rho_\lambda(h)\,d\mu(h)=0.
\]
Indeed, for any $g\in G$, by left invariance of Haar measure,
\[
\rho_\lambda(g)A_\lambda
=
\int_G \rho_\lambda(g)\rho_\lambda(h)\,d\mu(h)
=
\int_G \rho_\lambda(gh)\,d\mu(h)
=
\int_G \rho_\lambda(h)\,d\mu(h)
=
A_\lambda.
\]
Thus every vector in $\operatorname{Im}(A_\lambda)$ is fixed by $G$, so
\[
\operatorname{Im}(A_\lambda)\subseteq V_\lambda^G,
\]
where $V_\lambda^G:=\{v\in V_\lambda:\rho_\lambda(g)v=v\ \forall g\in G\}$.
Since $V_\lambda$ is irreducible and nontrivial, we have $V_\lambda^G=\{0\}$,
and therefore $A_\lambda=0$.

If $\lambda=\mathrm{triv}$, then $\rho_\lambda(h)=I_{V_\lambda}$ for all $h$, so $A_\lambda=I_{V_\lambda}$.
Therefore
\[
\int_G \Phi(T)(h)\,d\mu(h)
=
\langle B_{\mathrm{triv}},I_{V_{\mathrm{triv}}}\rangle_{\HS},
\]
and thus
\begin{equation} \label{eqn:psiexp}
\Psi(T)(g)
=
\sum_{\lambda\neq \mathrm{triv}}\langle B_\lambda,\rho_\lambda(g)\rangle_{\HS}.
\end{equation}
By Schur orthogonality (see \citep[Theorem~3.7]{Sepan}),
\[
\int_G
\langle B,\rho_\lambda(g)\rangle_{\HS}\,
\overline{\langle C,\rho_\eta(g)\rangle_{\HS}}
\,d\mu(g)
=
\sum_{i,j,k,\ell}
B_{ij}\,\overline{C_{k\ell}}
\int_G
\overline{(\rho_\lambda(g))_{ij}}\,
(\rho_\eta(g))_{k\ell}
\,d\mu(g)
\]
\[
=
\delta_{\lambda\eta}\,\frac{1}{d_\lambda}
\sum_{i,j} B_{ij}\,\overline{C_{ij}}
=
\delta_{\lambda\eta}\,\frac{1}{d_\lambda}\,\langle B,C\rangle_{\HS}.
\]

Applying this to the above expression for $\Psi(T)$ in \eqref{eqn:psiexp} gives
\begin{equation}\label{eqn:psi1}
\|\Psi(T)\|_{L^2(G,\mu)}^2
=
\sum_{\lambda\neq \mathrm{triv}}\frac{1}{d_\lambda}\,\|B_\lambda\|_{\HS}^2.
\end{equation}

We now compute the orthogonal projection onto $\mathcal C^\perp$.
First observe from Equation \eqref{eqn:LL} that for any $A\in\End(V_\lambda\otimes M_\lambda)$ and
$B\in\End(V_\lambda)$,
\[
\langle A,\;B\otimes I_{M_\lambda}\rangle_{\HS}
=
\langle L_\lambda(A),\,B\rangle_{\HS}.
\]
Therefore the Hilbert--Schmidt adjoint of $L_\lambda$ is
\[
L_\lambda^*(B)=B\otimes I_{M_\lambda},
\]
and hence
\[
(\ker L_\lambda)^\perp
=
\{B\otimes I_{M_\lambda}:B\in\End(V_\lambda)\}.
\]
We also have
\[
L_\lambda\!\left(\frac1{m_\lambda}B_\lambda\otimes I_{M_\lambda}\right)
=
\frac1{m_\lambda}\,m_\lambda\,B_\lambda
=
B_\lambda,
\]
so
\[
T_{\lambda\lambda}-\frac1{m_\lambda}B_\lambda\otimes I_{M_\lambda}\in \ker L_\lambda,
\]
while
\[
\frac1{m_\lambda}B_\lambda\otimes I_{M_\lambda}\in(\ker L_\lambda)^\perp.
\]
Hence
\[
\mathrm{Proj}_{(\ker L_\lambda)^\perp}(T_{\lambda\lambda})
=
\frac1{m_\lambda}B_\lambda\otimes I_{M_\lambda}.
\]
Now we finally identify $\mathcal C$.
From
\[
\Phi(T)(g)=\sum_{\lambda}\langle B_\lambda,\rho_\lambda(g)\rangle_{\HS},
\]
the function $\Phi(T)$ is constant on $G$ if and only if
\[
B_\lambda=0\qquad  \forall \lambda\neq \mathrm{triv}.
\]
Thus $\mathcal C$ consists exactly of those operators whose nontrivial diagonal
blocks lie in $\mathrm{ker} L_\lambda$, together with arbitrary off-diagonal blocks and an
arbitrary trivial-isotypic diagonal contribution. Consequently,
\[
\mathcal C^\perp
=
\bigoplus_{\lambda\neq \mathrm{triv}}
\{B\otimes I_{M_\lambda}:B\in\End(V_\lambda)\}.
\]
Therefore
\[
\mathrm{Proj}_{\mathcal C^\perp}T
=
\bigoplus_{\lambda\neq\mathrm{triv}}
\frac1{m_\lambda}B_\lambda\otimes I_{M_\lambda}.
\]

Taking Hilbert--Schmidt norms,
\begin{align}
\|\mathrm{Proj}_{\mathcal C^\perp}T\|_{\HS}^2
&=
\sum_{\lambda\neq\mathrm{triv}}
\left\|
\frac1{m_\lambda}B_\lambda\otimes I_{M_\lambda}
\right\|_{\HS}^2 =
\sum_{\lambda\neq \mathrm{triv}}
\frac1{m_\lambda^2}\,
\|B_\lambda\|_{\HS}^2\,
\|I_{M_\lambda}\|_{\HS}^2 \\
&=
\sum_{\lambda\neq\mathrm{triv}}
\frac1{m_\lambda}\,\|B_\lambda\|_{\HS}^2, \label{eqn:psiproj}
\end{align}
since $\|I_{M_\lambda}\|_{\HS}^2=m_\lambda$.

From \eqref{eqn:psi1} and \eqref{eqn:psiproj} we obtain
\begin{align}\label{eqn:psikey}
\|\Psi(T)\|_{L^2(G,\mu)}^2
&=
\sum_{\substack{\lambda\neq\mathrm{triv}\\ m_\lambda>0}}
\frac{m_\lambda}{d_\lambda}\,\frac1{m_\lambda}\|B_\lambda\|_{\HS}^2 \\
&\ge
\Big(\min_{\substack{\lambda\neq\mathrm{triv}\\ m_\lambda>0}}
\frac{m_\lambda}{d_\lambda}\Big)
\sum_{\substack{\lambda\neq\mathrm{triv}\\ m_\lambda>0}}
\frac1{m_\lambda}\|B_\lambda\|_{\HS}^2
=
\Big(\min_{\substack{\lambda\neq\mathrm{triv}\\ m_\lambda>0}}
\frac{m_\lambda}{d_\lambda}\Big)
\|\mathrm{Proj}_{\mathcal C^\perp}T\|_{\HS}^2.
\end{align}
Equality holds if and only if \(B_\lambda=0\) for every nontrivial \(\lambda\)
with
\[
\frac{m_\lambda}{d_\lambda}>
\min_{\substack{\eta\neq\mathrm{triv}\\ m_\eta>0}}
\frac{m_\eta}{d_\eta}.
\]
Finally, since $d_\lambda\le \dim U$ and $m_\lambda\ge 1$, this becomes
\[
\|\Psi(T)\|_{L^2(G,\mu)}^2
\ge
\frac1{\dim U}\,\|\mathrm{Proj}_{\mathcal C^\perp}T\|_{\HS}^2
\]
and the proof is done.
\end{proof}
\subsection{Equality cases and the Peter--Weyl--Plancherel picture} \label{SecEqualCases}

\anote{AMS: This section has some interesting sentences, but there are no theorem statements and no equations that get referenced. Should we delete it? If not, I think it would be helpful to say more explicitly what the target audience is supposed to get out of it. I DO think that it should be possible to extract the calculations that somebody like Alex Wein would be most interested in.]}

The proof of Theorem~\ref{thm:compact-group-variance} gives more than the universal
lower bound: it yields the exact identities 
\begin{equation}\label{eq:sharp-comparison}
\|\Psi(T)\|_{L^2(G,\mu)}^2
=
\sum_{\lambda\neq\mathrm{triv}}
\frac{1}{d_\lambda}\,\|B_\lambda\|_{\HS}^2,
\qquad
\|\mathrm{Proj}_{\mathcal C^\perp}T\|_{\HS}^2
=
\sum_{\lambda\neq\mathrm{triv}}
\frac{1}{m_\lambda}\,\|B_\lambda\|_{\HS}^2,
\end{equation}
where
\[
d_\lambda=\dim V_\lambda,
\qquad
m_\lambda=\dim M_\lambda.
\]

Equation \eqref{eq:sharp-comparison} is the only result in Subsection \ref{SecEqualCases} that is required for the main result in this paper, and so the remainder can be safely skipped on a first read-through. We include the rest of the subsection because we suspect that the general representation-theoretic estimates in this appendix may be of use to other statisticians working on the low-degree-polynomial heuristic, and we hope that the remarks in this section may help in generalizing our bounds beyond this specific problem.

Continuing from Equation \eqref{eq:sharp-comparison}, we note that the sharp constant produced by the argument is
\[
\kappa(U):=
\min_{\substack{\lambda\neq\mathrm{triv}\\ m_\lambda>0}}
\frac{m_\lambda}{d_\lambda},
\]
and we have
\[
\|\Psi(T)\|_{L^2(G,\mu)}^2
\ge
\kappa(U)\,\|\mathrm{Proj}_{\mathcal C^\perp}T\|_{\HS}^2.
\]
Equality holds if and only if
\[
B_\lambda=0
\qquad
\text{for every nontrivial }\lambda
\text{ such that }
\frac{m_\lambda}{d_\lambda}>\kappa(U).
\]
Thus the equality case is governed exactly by the multiplicity ratios
\[
\frac{m_\lambda}{d_\lambda}.
\]

The universal estimate in Theorem~\ref{thm:compact-group-variance} is obtained by
replacing \(\kappa(U)\) with the crude lower bound
\[
\kappa(U)\ge \frac{1}{\dim U},
\]
which follows from \(m_\lambda\ge 1\) and \(d_\lambda\le \dim U\).
This constant is sharp. Indeed, if \(U_\C\) is irreducible and nontrivial, then the
only nontrivial isotypic component has \(m_\lambda=1\) and \(d_\lambda=\dim U\), so
\[
\|\Psi(T)\|_{L^2(G,\mu)}^2
=
\frac{1}{\dim U}\,\|T\|_{\HS}^2
=
\frac{1}{\dim U}\,\|\mathrm{Proj}_{\mathcal C^\perp}T\|_{\HS}^2
\]
for every \(T\in \End(U_\C)\).

A second important equality pattern occurs when
\[
m_\lambda=d_\lambda
\qquad
\text{for every nontrivial irreducible }\lambda
\text{ occurring in }U_\C.
\]
Then the two identities in \eqref{eq:sharp-comparison} coincide term by term, and
hence
\[
\|\Psi(T)\|_{L^2(G,\mu)}^2
=
\|\mathrm{Proj}_{\mathcal C^\perp}T\|_{\HS}^2.
\]
This is exactly the multiplicity pattern of the regular representation, and it is at
this point that the present finite-dimensional argument meets classical harmonic
analysis.

Indeed, the passage from our finite-dimensional setting to \(L^2(G)\) is precisely
the passage from the finite-dimensional canonical decomposition
\citep[Theorem~2.24]{Sepan} to its Hilbert-space analogue
\citep[Theorem~3.19]{Sepan}.  By the Peter--Weyl theorem
\citep[Theorem~3.25 and Corollary~3.26]{Sepan},
\[
L^2(G)\cong \widehat{\bigoplus}_{[\pi]\in\widehat G} E_\pi^*\otimes E_\pi
\]
as a \(G\times G\)-module. In particular,
\[
\mathrm{Hom}_G(E_\pi,L^2(G))\cong E_\pi^*,
\qquad
m_\pi=\dim E_\pi^*=\dim E_\pi=d_\pi.
\]
Thus every irreducible \(E_\pi\) occurs in the regular representation with
multiplicity equal to its dimension.

Under the natural identification
\[
E_\pi^*\otimes E_\pi \cong \End(E_\pi),
\]
the Peter--Weyl decomposition becomes the operator-valued Fourier decomposition of
\(L^2(G)\).  In this form, the Plancherel theorem
\citep[Theorem~3.38]{Sepan} states that
\[
\|f\|_{L^2(G)}^2
=
\sum_{[\pi]\in\widehat G}
d_\pi\,\|\pi(f)\|_{\HS}^2.
\]
This is the infinite-dimensional analogue of the first identity in
\eqref{eq:sharp-comparison}: the quantities \(B_\lambda\) in our proof play the role
of operator-valued Fourier coefficients, and the weights \(1/d_\lambda\) are exactly
the Plancherel weights.

Equivalently, by Schur orthogonality, for every finite subset
\(F\subset \widehat G\),
\[
\left\|
\sum_{[\pi]\in F}
\langle B_\pi,\pi(\cdot)\rangle_{\HS}
\right\|_{L^2(G)}^2
=
\sum_{[\pi]\in F}
\frac{1}{d_\pi}\,\|B_\pi\|_{\HS}^2.
\]
This is precisely the Parseval identity for matrix coefficients.  In other words, one
of the identities produced by our argument is exactly the finite-dimensional
Peter--Weyl--Plancherel computation restricted to the coefficient function
\[
g\mapsto \Phi(T)(g)=\langle T,\rho(g)\rangle_{\HS}.
\]
The ratio \(m_\lambda/d_\lambda\) then measures how far the given representation
\(U_\C\) is from the regular representation: when \(m_\lambda=d_\lambda\) for all
relevant \(\lambda\), our comparison becomes termwise identical to the Plancherel
identity.

In the Abelian case \(G=S^1\), every irreducible representation is one-dimensional,
so \(d_\pi=1\) for all \(\pi\), and the Peter--Weyl--Plancherel picture reduces to the
usual Fourier series and Parseval identity on the circle; see
\citep[Theorem~3.43]{Sepan}.  Thus Theorem~\ref{thm:compact-group-variance}
should be viewed as a finite-dimensional shadow of classical harmonic analysis, with
the regular-representation equality case recovering exactly the classical
Peter--Weyl--Plancherel mechanism.

\subsection{Application to Polynomial Moments}
Let $X$ be Haar-distributed on $\So{n}$ and let $\mu=\haarOdist_{n}$ denote Haar probability measure.
All expectations $\mathbb{E}[\cdot]$ are with respect to $\mu$.

Fix an integer $d\ge 1$.
Let $\mathbb{R}^{[n]^d\times[n]^d}$ denote the space of real arrays
$T=(T_{\alpha,\beta})_{\alpha,\beta\in[n]^d}$, where
$\alpha=(\alpha_1,\dots,\alpha_d)$ and $\beta=(\beta_1,\dots,\beta_d)$.
For $T$ in this space define the homogeneous polynomial observable
\begin{equation}\label{eq:eval-map}
f_T(X) = \mathrm{Ev}_d(T)(X)
:=
\sum_{\alpha,\beta\in[n]^d} T_{\alpha,\beta}\prod_{t=1}^d X_{\alpha_t \beta_t}.
\end{equation}
Equip $\mathbb{R}^{[n]^d\times[n]^d}$ with the Frobenius inner product
\[
\langle T,S\rangle_F := \sum_{\alpha,\beta} T_{\alpha,\beta} S_{\alpha,\beta},
\qquad
\|T\|_F^2:=\langle T,T\rangle_F.
\]

Let
\[
\mathcal C_d:=\{T:\ \mathrm{Ev}_d(T)\text{ is constant $\mu$-a.s.}\}.
\]

Equivalently, $T\in\mathcal C_d$ iff $\Var(\mathrm{Ev}_d(T)(X))=0$.
Let $T^\sharp$ denote the orthogonal projection of $T$ onto $\mathcal C_d^\perp$.
Then
\begin{equation}\label{eq:centering}
\mathbb{E}[\mathrm{Ev}_d(T^\sharp)(X)]=0,
\qquad
\Var(\mathrm{Ev}_d(T)(X))=\Var(\mathrm{Ev}_d(T^\sharp)(X)).
\end{equation}

\begin{corollary}[Homogeneous case]
For homogeneous $T\in \mathbb{R}^{[n]^d\times [n]^d}$ define
$f_T(X)=\mathrm{Ev}_d(T)(X)$ as in \eqref{eq:eval-map}, and let $T^\sharp$ be the
Frobenius-orthogonal projection of $T$ onto $\mathcal C_d^\perp$. Then
\begin{equation}\label{eq:rep-homog-lb}
\Var(f_T(X)) \;\ge\; n^{-d}\,\|T^\sharp\|_F^2.
\end{equation}
\end{corollary}
\begin{proof}
Let
\[
U:=(\R^n)^{\otimes d},
\qquad
G=\So{n},
\]
and let
\[
\rho:G\to \End(U_\C)
\]
be the complexification of the tensor-power representation
$
\rho(g)=g^{\otimes d}.
$
We identify the coefficient array
\[
T=(T_{\alpha,\beta})_{\alpha,\beta\in[n]^d}\in \R^{[n]^d\times[n]^d}
\]
with the corresponding linear operator on \(U\) in the standard tensor-product basis:
if
\[
e_\alpha:=e_{\alpha_1}\otimes\cdots\otimes e_{\alpha_d},
\qquad
\alpha\in[n]^d,
\]
then \(T\) acts by
\[
T e_\beta=\sum_{\alpha\in[n]^d} T_{\alpha,\beta} e_\alpha.
\]
Under this identification, the Frobenius inner product on
\(\R^{[n]^d\times[n]^d}\) coincides with the Hilbert--Schmidt inner product on
\(\End(U)\), and hence
\[
\|T\|_F=\|T\|_{\HS}.
\]

Now let \(X\in \So{n}\). In the basis \(\{e_\alpha\}_{\alpha\in[n]^d}\), the matrix entries of
\(\rho(X)=X^{\otimes d}\) are
\[
\bigl(X^{\otimes d}\bigr)_{\alpha,\beta}
=
\prod_{t=1}^d X_{\alpha_t\beta_t}.
\]
Therefore
\[
\langle T,\rho(X)\rangle_{\HS}
=
\sum_{\alpha,\beta\in[n]^d} T_{\alpha,\beta}\bigl(X^{\otimes d}\bigr)_{\alpha,\beta}
=
\sum_{\alpha,\beta\in[n]^d} T_{\alpha,\beta}\prod_{t=1}^d X_{\alpha_t\beta_t}
=
\mathrm{Ev}_d(T)(X).
\]

Since $\mathrm{Ev}_d(T)(g)=\langle T,\rho(g)\rangle_{\HS}=\Phi(T)(g)$,
the subspace $\mathcal C_d$ coincides with $\mathcal C$, and hence
$T^\sharp$ is the orthogonal projection of $T$ onto $\mathcal C^\perp$.

By applying Theorem~\ref{thm:compact-group-variance} with this choice of
\(G\), \(U\), and \(\rho\), we obtain
\[
\Var(f_T(X)) = \Var(\mathrm{Ev}_d(T)(X))
=
\|\Psi(T)\|_{L^2(G,\mu)}^2
\ge
\frac{1}{\dim U}\,\|T^\sharp\|_F^2
=
n^{-d}\,\|T^\sharp\|_F^2
\]
 and the proof is finished. 
 \end{proof}
 \subsection{Inhomogeneous polynomials via grading}

Let $f$ be a polynomial function on $\So{n}$ that admits a representative of degree at most $d$
in the entries of $X$. Choose one such representative and decompose it as
\[
f = \sum_{k=0}^d f_k,
\]
where each $f_k$ is homogeneous of degree $k$.
(As a function on $\So{n}$, this decomposition need not be unique; here we fix one representative.)

To treat all degrees simultaneously, we introduce a graded representation.
Let
\[
\mathcal U:=\bigoplus_{k=0}^d (\R^n)^{\otimes k},
\qquad
\rho(g):=\bigoplus_{k=0}^d g^{\otimes k},
\qquad g\in \So{n}.
\]
Each summand $(\R^n)^{\otimes k}$ corresponds to the chosen homogeneous representative of
degree $k$, and the direct sum allows us to encode polynomials of mixed degree
as a single linear functional.

Arguing as in the homogeneous case, each $f_k$ can be written
as
\[
f_k(g)=\langle T_k,\,g^{\otimes k}\rangle_{\HS}
\]
for some $T_k\in \End((\R^n)^{\otimes k})$.
Setting
\[
T:=\bigoplus_{k=0}^d T_k \;\in\; \End(\mathcal U),
\]
we obtain the representation
\begin{equation}\label{eq:inhom-eval}
F_T(g)=\langle T,\rho(g)\rangle_{\HS}.
\end{equation}

Define, in direct analogy with the homogeneous case,
\[
\mathcal C_{\le d}
:=
\{T\in \End(\mathcal U): \langle T,\rho(g)\rangle_{\HS}
\text{ is constant $\mu$-a.s.}\}.
\]
Equivalently, $T\in \mathcal C_{\le d}$ if and only if
\[
\Var(\langle T,\rho(X)\rangle_{\HS})=0.
\]
Let
\begin{equation}\label{EqStandardizingProjection}
T^\sharp := \operatorname{Proj}_{\mathcal C_{\le d}^\perp}(T).
\end{equation}
Then \(T-T^\sharp \in \mathcal C_{\le d}\), so \(\langle T-T^\sharp,\rho(g)\rangle_{HS}\) is constant
\(\mu\)-a.s. Moreover,
\[
\mathbb{E}\big[\langle T^\sharp,\rho(X)\rangle_{HS}\big]=0.
\]
Hence
\begin{equation}
F_T(X)-\mathbb{E}[F_T(X)]
=
\langle T^\sharp,\rho(X)\rangle_{HS}
\qquad \mu\text{-a.s.}
\end{equation}
and therefore
\[
\Var(F_T(X))
=
\Var\!\big(\langle T^\sharp,\rho(X)\rangle_{HS}\big).
\]

\begin{corollary}[Inhomogeneous case]
\label{cor:inhomogeneous-case}\label{thm:inhom-rep}
Let $F_T$ be a polynomial in the entries of $X\in \So{n}$ of degree at most $d$,
and represent it as in \eqref{eq:inhom-eval}. Then
\[
\Var(F_T(X))
\ge
\frac{1}{(d+1)n^d}\,\|T^\sharp\|_{\HS}^2.
\]
\end{corollary}

\begin{proof}
We apply Theorem~\ref{thm:compact-group-variance} with
\[
G=\So{n},
\qquad
U=\mathcal U,
\qquad
\rho(g)=\bigoplus_{k=0}^d g^{\otimes k}.
\]
By construction \eqref{eq:inhom-eval}, we have
\[
F_T(g)=\Phi(T)(g)=\langle T,\rho(g)\rangle_{\HS}.
\]
Moreover, the space $\mathcal C_{\le d}$ defined above coincides with the space
$\mathcal C$ in Theorem~\ref{thm:compact-group-variance}, so that
$T^\sharp=\mathrm{Proj}_{\mathcal C^\perp}(T)$.

Therefore
\[
\Var(F_T(X))
=
\|\Psi(T)\|_{L^2(G,\mu)}^2
\ge
\frac{1}{\dim \mathcal U}\,\|T^\sharp\|_{\HS}^2.
\]
The dimension bounds follow from
\[
\dim \mathcal U=\sum_{k=0}^d n^k\le (d+1)n^d
\]
and the proof is finished.
\end{proof}

\section{Proof of Theorem \ref{thm:kac-jl-transform}} \label{SecProofFastJL}

We begin by setting some additional notation. For a polynomial $p$ in the matrix entries, write \be \label{EqJLSupNormDef}
\|p\|_{\infty,\So{n}}=\sup_{Q\in\So{n}}|p(Q)|.
\ee
Given parameters $d_{\mathrm{ph}},L>0$, say that a probability measure $\mu$ on $\So{n}$ is \emph{$(d_{\mathrm{ph}},L)$-pseudo-Haar in supnorm} if, for every real polynomial $p$ in the matrix entries with $\deg(p)\leq d_{\mathrm{ph}}$ and $\|p\|_{\infty,\So{n}}\leq1$,
\be \label{EqJLPseudoHaarDef}
\left|\int p\,d\mu-\int p\,d\haarOdist_{n}\right|\leq n^{-L}.
\ee

Our main estimate is the following proof that being pseudo-Haar is good enough to obtain the usual Johnson-Lindenstrauss estimate:

\begin{theorem} \label{thm:jl-pseudohar-transfer}
There is a universal constant $C<\infty$ with the following property. Let $Q$ be an $\So{n}$-valued random matrix whose law is $(d_{\mathrm{ph}},L)$-pseudo-Haar in the sense of \eqref{EqJLPseudoHaarDef}. Let $\mathcal X\subset\mathbb{R}^n$ have $N\geq2$ points, and let $0<\varepsilon,\delta<1$. Set
\be \label{EqJLPHRChoice}
r=\left\lceil C\log\frac{N}{\delta}\right\rceil.
\ee
Assume
\be \label{EqJLPHAssumptions}
4r\leq d_{\mathrm{ph}},\qquad
\ell\geq C\varepsilon^{-2}r,
\qquad
L\log n\geq \log\frac{2\binom{N}{2}}{\delta}+2r\log\frac{n}{\varepsilon\ell}.
\ee
Then, with probability at least $1-\delta$ over $Q$,
\be \label{EqJLPHConclusion}
(1-\varepsilon)\|x_i-x_j\|_2^2
\leq
\|A_Q(x_i-x_j)\|_2^2
\leq
(1+\varepsilon)\|x_i-x_j\|_2^2
\ee
for every pair $1\leq i<j\leq N$.
\end{theorem}

The proof uses the following standard Haar moment estimate. We will use essentially the same argument as in  \cite{DasguptaGupta2003ElementaryProofJL}, except that we replace exponential moments with merely the moderate-degree polynomial moments that are actually required.

\begin{lemma}[Haar moment bound] \label{lem:jl_haar_moment}
There are universal constants $C_0,c_0>0$ such that, for every unit vector $v\in \mathbb{S}^{n-1}$ and every integer $1\leq r\leq c_0\ell$,
\be \label{EqJLHaarMoment}
\mathbb{E}_{H\sim\haarOdist_{n}}|Z_v(H)|^{2r}\leq\left(C_0\frac{r}{\ell}\right)^r.
\ee
\end{lemma}

\begin{proof}
By Haar invariance of $\haarOdist_{n}$, the vector $Hv$ is uniformly distributed on $\mathbb{S}^{n-1}$. Therefore $\|\Pi_\ell Hv\|_2^2$ has the same distribution as the sum of the first $\ell$ squared coordinates of a uniform random point on $\mathbb{S}^{n-1}$, equivalently a beta random variable with parameters $\ell/2$ and $(n-\ell)/2$.

The standard beta, or spherical-cap, concentration estimate gives universal constants $a,b>0$ such that
\be \label{EqJLHaarTailSmall}
\mathbb{P}(|Z_v(H)|>t)\leq 2\exp(-a\ell t^2),\qquad 0<t\leq1,
\ee
and
\be \label{EqJLHaarTailLarge}
\mathbb{P}(|Z_v(H)|>t)\leq 2\exp(-b\ell t),\qquad t\geq1.
\ee
Using
\be \label{EqJLTailIntegration}
\mathbb{E} |Z_v(H)|^{2r}=\int_0^\infty 2r t^{2r-1}\mathbb{P}(|Z_v(H)|>t)\,dt,
\ee
we bound the contribution from $0<t\leq1$ by extending the integral to infinity and applying \eqref{EqJLHaarTailSmall}:
\[
\int_0^\infty 4r t^{2r-1}e^{-a\ell t^2}\,dt=2r(a\ell)^{-r}\Gamma(r)\leq \left(C\frac r\ell\right)^r.
\]
Similarly, \eqref{EqJLHaarTailLarge} gives
\[
\int_1^\infty 4r t^{2r-1}e^{-b\ell t}\,dt\leq \left(C\frac r\ell\right)^{2r}\leq \left(C_0\frac r\ell\right)^r,
\]
after choosing $c_0>0$ sufficiently small and $C_0<\infty$ sufficiently large. Combining these estimates with \eqref{EqJLTailIntegration} proves \eqref{EqJLHaarMoment}.
\end{proof}

\begin{prop}[Moment-transfer form] \label{prop:jl_moment_transfer}
Let $Q$ be an $\So{n}$-valued random matrix, let $H\sim\haarOdist_{n}$, and let $\mathcal X=\{x_1,\ldots,x_N\}\subset\mathbb{R}^n$. Let $0<\varepsilon<1$ and let $r$ be a positive integer with $r\leq c_0\ell$, where $c_0$ is the constant in Lemma \ref{lem:jl_haar_moment}. Suppose that, for every $v\in\mathcal V_{\mathcal X}$,
\be \label{EqJLMomentClosenessAlpha}
\left|\mathbb{E} Z_v(Q)^{2r}-\mathbb{E} Z_v(H)^{2r}\right|\leq \alpha.
\ee
Then
\begin{align}
&\mathbb{P}\left(\exists\,i<j:\left|\frac{\|A_Q(x_i-x_j)\|_2^2}{\|x_i-x_j\|_2^2}-1\right|>\varepsilon\right) \notag\\
&\hspace{2cm}\leq
\binom{N}{2}\left[\left(C_0\frac{r}{\varepsilon^2\ell}\right)^r+\frac{\alpha}{\varepsilon^{2r}}\right]. \label{EqJLMomentTransferFailure}
\end{align}
\end{prop}

\begin{proof}
For fixed $v\in\mathcal V_{\mathcal X}$, Markov's inequality, \eqref{EqJLMomentClosenessAlpha}, and Lemma \ref{lem:jl_haar_moment} give
\[
\mathbb{P}(|Z_v(Q)|>\varepsilon)
\leq \varepsilon^{-2r}\mathbb{E}|Z_v(Q)|^{2r}
\leq \left(C_0\frac{r}{\varepsilon^2\ell}\right)^r+\frac{\alpha}{\varepsilon^{2r}}.
\]
Taking a union bound over the at most $\binom{N}{2}$ vectors in $\mathcal V_{\mathcal X}$ proves \eqref{EqJLMomentTransferFailure}.
\end{proof}

\begin{proof}[Proof of Theorem \ref{thm:jl-pseudohar-transfer}]
Fix $v\in \mathbb{S}^{n-1}$ and define
\be \label{EqJLMomentPolynomial}
P_{v,r}(G)=Z_v(G)^{2r}.
\ee
Since $Z_v(G)$ has degree $2$ in the entries of $G$, \eqref{EqJLMomentPolynomial} has degree at most $4r$. Moreover, for every $Q\in\So{n}$,
\be \label{EqJLZSupBound}
0\leq\|\Pi_\ell Qv\|_2^2\leq1,
\qquad
|Z_v(Q)|=\left|\frac n\ell\|\Pi_\ell Qv\|_2^2-1\right|\leq \frac n\ell.
\ee
Thus
\be \label{EqJLMomentSupNorm}
\|P_{v,r}\|_{\infty,\So{n}}\leq \left(\frac n\ell\right)^{2r}.
\ee
By \eqref{EqJLPHAssumptions}, the normalized polynomial $(\ell/n)^{2r}P_{v,r}$ has degree at most $d_{\mathrm{ph}}$ and supnorm at most $1$. Applying \eqref{EqJLPseudoHaarDef} gives
\be \label{EqJLPHMomentTransfer}
\left|\mathbb{E} P_{v,r}(Q)-\mathbb{E} P_{v,r}(H)\right|\leq n^{-L}\left(\frac n\ell\right)^{2r}.
\ee
Apply Proposition \ref{prop:jl_moment_transfer} with $\alpha=n^{-L}(n/\ell)^{2r}$. After increasing the universal constant $C$ in \eqref{EqJLPHRChoice} and \eqref{EqJLPHAssumptions}, the condition $\ell\geq C\varepsilon^{-2}r$ makes the first term in \eqref{EqJLMomentTransferFailure} at most $\delta/2$. The last inequality in \eqref{EqJLPHAssumptions} makes the second term in \eqref{EqJLMomentTransferFailure} at most $\delta/2$. Therefore the bad event has probability at most $\delta$, and \eqref{EqJLPairDistortion} gives \eqref{EqJLPHConclusion}.
\end{proof}

\begin{proof}[Proof of Theorem \ref{thm:kac-jl-transform}]
Let $H\sim\haarOdist_{n}$ and set $d=4r$. We first verify the moment-transfer hypothesis \eqref{EqJLMomentClosenessAlpha} for $Q=M$.

For a monomial $\mathfrak m$ of degree $m\leq d$, Theorem \ref{LemmaDiffUpperBoundMonomial} and the definition \eqref{eq:omega_def} give, after increasing universal constants if necessary,
\be \label{EqJLKacMonomialDifference}
\left|\mathbb{E}_{M\sim\kacdist_{n,T,x}} \mathfrak m(M)-\mathbb{E}_{H\sim\haarOdist_{n}}\mathfrak m(H)\right|
\leq C d\, n^{-s},
\ee
where
\be \label{EqJLKacSDef}
s=\left\lfloor \frac{T}{C n(d+\log n)\log n}\right\rfloor.
\ee
Indeed, when $m\leq d$, the monomial $\mathfrak m$ depends on at most $d$ columns, and the bound \eqref{EqJLKacMonomialDifference} is the degree-$d$ version of Theorem \ref{LemmaDiffUpperBoundMonomial} with constants enlarged to dominate all lower degrees.

Next expand \eqref{EqJLZvDef} as a polynomial in the entries of $G$:
\be \label{EqJLZCoeffExpansion}
Z_v(G)=\frac n\ell\sum_{a=1}^{\ell}\sum_{b,c=1}^{n}v_bv_cG_{ab}G_{ac}-1.
\ee
Since $\|v\|_1^2\leq n\|v\|_2^2=n$, \eqref{EqJLZCoeffExpansion} gives
\be \label{EqJLZCoeffOneBound}
\|Z_v\|_{\mathrm{coeff},1}\leq 1+n\|v\|_1^2\leq 1+n^2\leq 2n^2.
\ee
Therefore
\be \label{EqJLMomentCoeffOneBound}
\|Z_v^{2r}\|_{\mathrm{coeff},1}\leq (2n^2)^{2r}.
\ee
Combining \eqref{EqJLKacMonomialDifference} and \eqref{EqJLMomentCoeffOneBound}, we obtain, uniformly over $v\in \mathbb{S}^{n-1}$,
\be \label{EqJLKacMomentAlpha}
\left|\mathbb{E} Z_v(M)^{2r}-\mathbb{E} Z_v(H)^{2r}\right|
\leq C r (2n^2)^{2r} n^{-s}.
\ee
By increasing $C_{\mathrm{JL}}$ in \eqref{EqJLKacEllChoice} and \eqref{EqJLKacTimeAssumption}, the condition \eqref{EqJLKacEllChoice} makes $r\leq c_0\ell$ and makes the first term in \eqref{EqJLMomentTransferFailure} at most $\delta/2$. Writing $B=r\log(n/\varepsilon)+\log(N/\delta)$, the time condition \eqref{EqJLKacTimeAssumption} gives $s\log n\geq cC_{\mathrm{JL}}B-\log n$ for a universal $c>0$; since $B\geq\log n$, increasing $C_{\mathrm{JL}}$ absorbs the final $\log n$ loss. Thus \eqref{EqJLKacTimeAssumption}, together with \eqref{EqJLKacMomentAlpha}, makes the second term in \eqref{EqJLMomentTransferFailure} at most $\delta/2$. Applying Proposition \ref{prop:jl_moment_transfer} with $Q=M$ gives
\[
\mathbb{P}\left(\exists\,i<j:\left|\frac{\|A_M(x_i-x_j)\|_2^2}{\|x_i-x_j\|_2^2}-1\right|>\varepsilon\right)
\leq\delta.
\]
Using \eqref{EqJLPairDistortion}, this is exactly \eqref{EqJLKacConclusion}. Pairs with $x_i=x_j$ satisfy \eqref{EqJLKacConclusion} trivially.

If $N/\delta\leq n^A$ and $\varepsilon$ is fixed, then $r=O_A(\log n)$. The conditions \eqref{EqJLKacEllChoice} and \eqref{EqJLKacTimeAssumption} then reduce to $\ell=O_A(\varepsilon^{-2}\log n)$ and $T\geq C_A n\log^3 n$, after increasing $C_A$.
\end{proof}

\end{document}